\documentclass[letterpaper,11pt]{article}
\usepackage[utf8]{inputenc}
\usepackage{authblk}
\usepackage{mathtools}

\renewcommand{\geq}{\geqslant}
\renewcommand{\leq}{\leqslant}

\usepackage{amssymb,amsthm,amscd,latexsym,mathrsfs,units,enumerate,bm,bbm,cancel,physics}
\usepackage[colorlinks]{hyperref}
\usepackage{graphicx,wrapfig}
\graphicspath{{./images/}}

\usepackage[margin=1in]{geometry}

\theoremstyle{plain}
\newtheorem{theorem}{Theorem}
\newtheorem{lemma}[theorem]{Lemma}

\newtheorem{corollary}[theorem]{Corollary}
\newtheorem*{thm*}{Theorem}

\theoremstyle{definition} 

\theoremstyle{remark}
\newtheorem*{remark*}{Remark}

\numberwithin{equation}{section}

\newtheorem*{prob*}{}

\newcommand{\R}{\mathbb{R}}

\DeclareMathOperator{\disc}{disc}

\DeclareMathOperator{\cn}{cn}

\DeclareMathOperator{\sn}{sn}

\title{Superharmonic instability for regularized long-wave models}

\author[1]{Jared~C.~Bronski\thanks{E-mail:~bronski@illinois.edu.}}
\author[1]{Vera~Mikyoung~Hur\thanks{E-mail:~verahur@math.uiuc.edu. VMH is supported by NSF DMS-2009981.}}
\author[1]{Samuel~Lee~Wester\thanks{E-mail:~swester3@illinois.edu}}
\affil[1]{Department of Mathematics, University of Illinois at Urbana-Champaign \protect\\ 
Urbana, IL 61801, USA}

\begin{document}

\maketitle

\begin{abstract}
We examine the spectral stability and instability of periodic traveling waves for regularized long-wave models. Examples include the regularized Boussinesq, Benney--Luke, and Benjamin--Bona--Mahony equations. Of particular interest is a striking new instability phenomenon---spectrum off the imaginary axis extending into infinity. The spectrum of the linearized operator of the generalized Korteweg--de Vries equation, for instance, lies along the imaginary axis outside a bounded set. The spectrum for a regularized long-wave model, by contrast, can vary markedly with the parameters of the periodic traveling waves. We carry out asymptotic spectral analysis to short wavelength perturbations, distinguishing whether the spectrum tends to infinity along the imaginary axis or some curve whose real part is nonzero. We conduct numerical experiments to corroborate our analytical findings. 
\end{abstract}


\section{Introduction}

The Korteweg--de Vries (KdV) equation~\cite{Bou,KdV}
\begin{equation}\label{eqn:KdV} 
u_t+u_x+u_{xxx}+(u^2)_x=0
\end{equation}
and the Boussinesq equation~\cite{Bou}
\begin{equation}\label{eqn:Bou}
u_{tt}-u_{xx}-u_{xxxx}-(u^2)_{xx}=0
\end{equation}
are basic mathematical models for small amplitude and relatively long waves in water, under the influence of gravity and possibly surface tension. Here $t\in\mathbb{R}$ is proportional to elapsed time, $x\in\mathbb{R}$ is related to the spatial variable in the direction of wave propagation, and $u(x,t)$ is real valued, describing the fluid surface or a velocity. 

There are many variants of \eqref{eqn:KdV} and \eqref{eqn:Bou}. Under the assumption that $u_t+u_x$ is of smaller order, Benjamin, Bona and Mahony~\cite{BBM} proposed 
\begin{equation*}\label{eqn:BBM}
u_t+u_x-u_{xxt}+(u^2)_x=0
\end{equation*}
as an alternative to \eqref{eqn:KdV}. Also under the assumption that $u_t\pm u_x$ are of smaller order, the so-called regularized Boussinesq equation\footnote{This does not appear explicitly in \cite{Bou}. But \cite[(280)]{Bou}, for instance, after several `higher order terms' drop out, becomes equivalent to what appears in \cite{Whitham}.} 
\begin{equation}\label{eqn1:rBou}
u_{tt}-u_{xx}-u_{xxtt}-(u^2)_{xx}=0
\end{equation}
is at least formally equivalent to \eqref{eqn:Bou} for long waves. Indeed, the dispersion relation of \eqref{eqn1:rBou}, 
\[
\omega^2(k)=\frac{k^2}{1+k^2}=k^2(1-k^2+O(k^4))\quad\text{as $|k|\to0$},
\]
agrees with the dispersion relation of \eqref{eqn:Bou} up to the order of $k^4$ when $|k|$ is small, but \eqref{eqn1:rBou} is preferable to \eqref{eqn:Bou} for short and intermediately long waves because the Cauchy problem for \eqref{eqn:Bou} is ill-posed. 

Another variant is the Benney--Luke equation~\cite{BL}
\begin{equation}\label{eqn1:BL}
u_{tt}-u_{xx}+au_{xxxx}-bu_{xxtt}+u_tu_{xx}+2u_xu_{xt}=0,
\end{equation}
where\footnote{$a=\frac16$ and $b=\frac12$ in \cite{BL}} $a,b \geq 0$ satisfy $a-b=\textrm{Bo}^{-1}-\frac13$ and $\textrm{Bo}^{-1}$ is the inverse of the Bond number, describing surface tension strength. The dispersion relation of \eqref{eqn1:BL} is
\[
\omega^2(k)=k^2\frac{1+ak^2}{1+bk^2}
\]
and remains bounded for all $k\in\mathbb{R}$, provided that $a=0$ and $b\neq0$.
Also Bona, Chen and Saut~\cite{BCS,BCS2} proposed a four parameter family of Boussinesq systems 
\begin{equation}\label{eqn1:BCS}
\begin{aligned}
&\eta_t+u_x+au_{xxx}-b\eta_{xxt}+(\eta u)_x=0, \\
&u_t+\eta_x+c\eta_{xxx}-du_{xxt}+uu_x=0, 
\end{aligned}
\end{equation}
where $a,b,c,d\geq0$ satisfy $a+b=\frac12(\theta^2-\frac13)$ and $c+d=\frac12(1-\theta^2)$ for some $\theta\in[0,1]$.
The dispersion relation of \eqref{eqn1:BCS} is 
\[
\omega^2(k)=k^2\frac{(1-ak^2)(1-ck^2)}{(1+bk^2)(1+dk^2)}
\]
and remains bounded for all $k\in\mathbb{R}$, provided that $a$ or $c=0$ while $b,d\neq0$.

For the generalized KdV equation 
\begin{equation}\label{eqn:gKdV}
u_t+u_x+u_{xxx}+(f(u))_x=0,\quad\text{$f$ is some nonlinearity}, 
\end{equation}
the spectrum of the linearized operator about a periodic traveling wave lies along the imaginary axis outside of a bounded set. 
See Appendix~\ref{appn} for a proof. 
This is in some sense implicit in many numerical studies of stability and instability. See \cite{BHJ,BJ,BJK}, among others, for analytical studies of stability and instability. We expect that the same holds true for \eqref{eqn1:BL} and \eqref{eqn1:BCS} so long as $a,c\neq0$. The proof in Appendix~\ref{appn} hinges on the fact that $|\omega(k+1)-\omega(k)|\to\infty$ as $|k|\to\infty$, where $\omega(k)=k^3-k$ is the dispersion relation of \eqref{eqn:gKdV}. 
For regularized long-wave models such as \eqref{eqn1:rBou}, and \eqref{eqn1:BL}, where $a=0$ while $b\neq0$, and \eqref{eqn1:BCS}, where $a$ or $c=0$ while $b,d\neq0$, by contrast, $|\omega(k+1)-\omega(k)|$ remains bounded or even vanishes as $|k|\to\infty$. 

Indeed, the asymptotics of the spectrum at infinity of the linearized operator of a regularized long-wave model can vary markedly with the parameters of the periodic traveling waves. Figure~\ref{FourthBandAndGapExample} provides examples for two periodic traveling waves of \eqref{eqn1:rBou}, for which only one of the parameters differs less than $4\%$. On the left, the spectrum lies along the imaginary axis outside of a bounded set. 
On the right, on the other hand, the spectrum tends towards infinity along the dashed lines whose real part is nonzero. Indeed, the periodic traveling wave is spectrally unstable to arbitrarily short wavelength perturbations, the opposite to modulational instability to arbitrarily long wavelength perturbations. The aim here is to explain such striking and new instability for \eqref{eqn1:rBou}, and \eqref{eqn1:BL}, where $a=0$ while $b\neq0$, and \eqref{eqn1:BCS}, where $a,c=0$ while $b,d\neq0$. Modulational instability will be addressed in a companion article \cite{BHW2}. 

In Section~\ref{sec:Boussinesq} we set forth asymptotic spectral analysis to short wavelength perturbations for \eqref{eqn1:rBou}, demonstrating that: if an associated Hill's differential equation (see \eqref{eqn2:Hill}) is elliptic (in a band), depending on the parameters of the periodic traveling wave, 
then the spectrum tends to infinity along the imaginary axis. If \eqref{eqn2:Hill} is hyperbolic (in a gap), on the other hand, then the spectrum tends to infinity along a line whose real part is nonzero. Actually, \eqref{eqn2:Hill} is a three-gap Lam\'e equation, whose band edges can be found in closed form, enabling us to classify the spectrum at infinity for all periodic traveling waves. We compute the spectrum numerically to corroborate the analytical predictions. 

In Section~\ref{sec:BL} we turn our attention to \eqref{eqn1:BL}, where $a=0$ and, for simplicity of notation, $b=1$. The associated Lam\'e equation (see \eqref{eqn3:Lame}) is no longer explicitly solvable. Nevertheless we can compute the band edges numerically, distinguishing whether the spectrum tends to infinity along the imaginary axis or some line whose real part is nonzero. 

Last but not least, in Section~\ref{sec:abcd}, we turn to \eqref{eqn1:BCS}, where $a=c=0$ and\footnote{for convenience and to better corresponds to earlier works \cite{CCN,CCD}} $b=d=\frac16$, and establish that: if the mean of $1+\eta$ over the period is negative, depending on the parameters of the periodic traveling wave, then the spectrum tends to infinity along some curve whose real part is $O(k^{-1})$ for  $|k|\gg1$, that is, short wavelength perturbations. Numerical results are in excellent agreement with analytical ones. 

\section{The regularized Boussinesq equation}\label{sec:Boussinesq}

We begin by rewriting \eqref{eqn1:rBou} as 
\[
u_t=(1-\partial_x^2)^{-1}v_x \quad\text{and}\quad v_t=(u+u^2)_x
\]
or, equivalently \cite{HP:BBM},
\begin{equation}\label{eqn2:uv}
\mqty(u \\ v)_t=
\mqty(0 & \partial_x \\ \partial_x & 0)\mqty(u+u^2 \\ (1-\partial_x^2)^{-1}v).
\end{equation}
Throughout the section we employ the notation $\mathbf{u}=\mqty(u \\ v)$. We remark that \eqref{eqn2:uv} is in the Hamiltonian form
\[
\mathbf{u}_t=J\delta H(\mathbf{u}),
\]
where $J=\mqty(0 & \partial_x \\ \partial_x & 0)$ is the symplectic form, 
\[
H(\mathbf{u})=\int \left(\frac{1}{2}u^2+\frac{1}{3}u^3+\frac{1}{2}v(1-\partial_x^2)^{-1}v\right)~\dd{x}
\]
is the Hamiltonian, and $\delta$ denotes variational differentiation. In addition to $H$, \eqref{eqn2:uv} has three conserved quantities
\begin{align*}
&P(\mathbf{u})=\int uv~\dd{x}&&\text{(momentum)}, \\
&M_1(\mathbf{u})=\int u~\dd{x}\quad\text{and}\quad
M_2(\mathbf{u})=\int v~\dd{x} &&\text{(masses)}.
\end{align*}

\subsection{Parametrization of periodic traveling waves}\label{sec2:periodic}

A traveling wave of \eqref{eqn2:uv} takes the form $\mathbf{u}(x-ct-x_0)$, where $c\neq0,\in\mathbb{R}$ is the wave speed and $x_0\in\mathbb{R}$ the spatial translate, and it arises as a critical point of the `augmented Hamiltonian'
\[
H_\text{aug}=H+cP+b_1M_1+b_2M_2
\]
for some $b_1$, $b_2\in\R$. That is,
\begin{equation}\label{eqn2:uv0}
\delta H_\text{aug}(\mathbf{u})
=\mqty(u+u^2+cv+b_1 \\ (1-\partial_x^2)^{-1}v+cu+b_2)=\boldsymbol{0}.
\end{equation}
Eliminating $v$ from \eqref{eqn2:uv0}, we arrive at
\begin{equation}\label{eqn2:u0}
c^2u''+(1-c^2)u+u^2+b_1-b_2c=0.
\end{equation}
Here and elsewhere, the prime denotes ordinary differentiation. Multiplying \eqref{eqn2:u0} through by $u'$ and integrating, moreover,
\begin{equation}\label{eqn2:E-V}
\frac12c^2(u')^2=E-V(u;c,b_1,b_2),\quad V(u;c,b_1,b_2)=\frac{1}{3}u^3+\frac12(1-c^2)u^2+(b_1-b_2c)u,
\end{equation}
for some $E\in\R$. One can perform phase plane analysis for the existence of non-constant periodic solutions of \eqref{eqn2:E-V} and, hence, non-constant periodic traveling waves of \eqref{eqn2:uv}, depending on $c$, $b_1$, $b_2$, $E\in\R$, or depending on $\alpha$, $\beta$, $\gamma\in\R$ such that $\alpha<\beta<\gamma$, the roots of the cubic polynomial $E-V(u;c,b_1,b_2)$, provided that they exist. A straightforward calculation reveals that
\begin{equation}\label{eqn2:cbE->abc}
E=\frac13\alpha\beta\gamma,\quad 
b_1-b_2c=\frac13(\alpha\beta+\beta\gamma+\alpha\gamma)
\quad\text{and}\quad c^2-1=\frac23(\alpha+\beta+\gamma).
\end{equation}
Since $c\neq0$, $\alpha+\beta+\gamma>-\frac32$ must hold true. Let
\begin{equation}\label{def2:triangle}
\triangle=\Big\{(\alpha,\beta,\gamma)\in\R^3: \alpha<\beta<\gamma\quad\text{and}\quad\alpha+\beta+\gamma>-\frac32\Big\}.
\end{equation}
Figure~\ref{BandsAndGaps} shows $\triangle$ in the $(\alpha,\beta)$ plane when $\gamma=1$. 

A traveling wave of \eqref{eqn2:uv} depends additionally on $x_0$. On the other hand, \eqref{eqn2:uv} remains invariant under the translation of the $x$ axis, whereby we can mod out $x_0$, requiring $\mathbf{u}'(0)=\mathbf{0}$, that is, $\mathbf{u}$ is even. 

Actually, periodic solutions of \eqref{eqn2:u0} can be found in closed form in terms of the Jacobi elliptic functions. Recall that $y(x)=\sn(x,\sqrt{m})$ is a solution of
\begin{align*}
y''+(1+m)y=&2my^3,
\intertext{where $m\in(0,1)$ is the elliptic parameter, and $z(x)=\sn^2(x,\sqrt{m})$ is a solution of}
z''+4(1+m)z=&2+6mz^2.
\end{align*}
Throughout we work with the elliptic parameter rather than the elliptic modulus $k$, where $m=k^2$. Let $(\alpha,\beta,\gamma)\in\triangle$ and our task is to solve \eqref{eqn2:u0} or, equivalently,
\begin{equation}\label{eqn2:u0'}
u''-\frac{2(\alpha+\beta+\gamma)}{2(\alpha+\beta+\gamma)+3}u+ \frac{3}{2(\alpha+\beta+\gamma)+3}u^2 +\frac{\alpha\beta+\beta\gamma+\alpha\gamma}{2(\alpha+\beta+\gamma)+3}=0
\end{equation}
by \eqref{eqn2:cbE->abc}, subject to, say,
\[
u(0)=\gamma\quad\text{and}\quad u'(0)=0.
\]
Trying
\begin{equation}\label{def2:sn}
u(x)=\gamma-(\gamma-\beta)\sn^2(ax,\sqrt{m}),
\end{equation}
after some algebra we find
\begin{equation}\label{def2:am}
m=\frac{\gamma-\beta}{\gamma-\alpha}\quad\text{and}\quad
a=\sqrt{\frac{\gamma-\alpha}{4(\alpha+\beta+\gamma)+6}}.
\end{equation}
Clearly, $0<m<1$ and $a>0$. The period of \eqref{def2:sn} is 
\begin{equation}\label{def2:T}
T=\frac{2K(\sqrt{m})}{a},
\end{equation}
where 
\[
K(\sqrt{m}) = \int_0^1 \frac{\dd{s}}{\sqrt{(1-s^2)(1-ms^2)}}
\]
is the complete elliptic integral of the first kind.

To summarize, whenever $(\alpha,\beta,\gamma)\in\triangle$, \eqref{def2:sn} gives a periodic traveling wave of \eqref{eqn2:uv}, depending on $\alpha$, $\beta$, $\gamma$ by \eqref{def2:am}, where $v$ is in \eqref{eqn2:uv0}.

\begin{remark*}
More generally, consider 
\begin{equation}\label{eqn:fBou}
u_{tt}-u_{xx}-u_{xxtt}-f(u)_{xx}=0,\quad\text{$f$ is some nonlinearity}.
\end{equation}
When $f(u)=u^2$, \eqref{eqn:fBou} becomes \eqref{eqn1:rBou}. 
Proceeding as above, one can deduce that a traveling wave of \eqref{eqn:fBou} satisfies
\begin{equation}\label{eqn:fE-V}
\frac12c^2(u')^2=E-V(u;c,b_1,b_2),\quad V(u,c,b_1,b_2)=F(u)+\frac12(1-c^2)u^2+(b_1-b_2c)u,
\end{equation}
for some $c\neq0,\in\mathbb{R}$ for some $b_1$, $b_2$, $E\in\mathbb{R}$, where $F'=f$. If $E-V(u;c,b_1,b_2)$ has simple roots $u_\pm$, depending on $c$, $b_1$, $b_2$, $E$, such that $u_-<u_+$ and if $E-V(u;c,b_1,b_2)>0$ for $u\in(u_-,u_+)$ then \eqref{eqn:fE-V} has a non-constant periodic solution, whose period is 
\[
T=\sqrt{2}\int^{u_+}_{u_-}\frac{c~\dd{u}}{\sqrt{E-V(u;c,b_1,b_2)}}
=\frac{\sqrt{2}}{2}\oint_\varGamma\frac{c~\dd{u}}{\sqrt{E-V(u;c,b_1,b_2)}},
\]
where $\varGamma$ is a Jordan curve in the complex $u$ plane containing the interval $[u_-,u_+]$  in the interior region. One can treat the contour integration at the branch points in the usual way. 

All of what follows would succeed with \eqref{eqn:fBou} mutatis mutandis. One does not expect that periodic traveling waves of \eqref{eqn:fBou} can be found in closed form, except for few nonlinearities, but their parametrization by $c$, $b_1$, $b_2$, $E$ suffices. 
Here we focus our attention to $f(u)=u^2$, nevertheless, because such a nonlinearity is characteristic of many wave phenomena and, moreover, leads to simple analytic formulae. 
\end{remark*}

\subsection{Asymptotic spectral analysis to short wavelength perturbations}\label{sec2:spec}

Let $(\alpha,\beta,\gamma)\in\triangle$ (see \eqref{def2:triangle}) and $\mathbf{u}$ denotes a periodic traveling wave of \eqref{eqn2:uv} (see \eqref{eqn2:uv0}, \eqref{def2:sn}, \eqref{def2:am}), $c$ the wave speed (see the third equation of \eqref{eqn2:cbE->abc}) and $T$ the period (see \eqref{def2:T}), depending on $\alpha$, $\beta$, $\gamma$. Linearizing \eqref{eqn2:uv} about $\mathbf{u}$ in the frame of reference moving at the speed $c$, we arrive at
\[
\boldsymbol{\phi}_t=J\delta^2H_\text{aug}(\mathbf{u})\boldsymbol{\phi}
=\mqty(0 & \partial_x \\ \partial_x & 0)
\mqty(1+2u & c \\ c & (1-\partial_x^2)^{-1})\boldsymbol{\phi}.
\]
Seeking a solution of the form $\boldsymbol{\phi}(x,t)=e^{\lambda t}\boldsymbol{\phi}(x)$,  $\lambda\in\mathbb{C}$, we arrive at 
\begin{equation}\label{eqn2:spec}
\lambda\boldsymbol{\phi}
=\mqty(c\partial_x & \partial_x(1-\partial_x^2)^{-1} \\ \partial_x(1+2u) & c\partial_x)\boldsymbol{\phi}=:\mathbf{L}(\mathbf{u})\boldsymbol{\phi},
\end{equation}
where $\mathbf{L}(\mathbf{u}):H^1(\R)\times H^1(\R)\subset L^2(\R)\times L^2(\R)\to L^2(\R)\times L^2(\R)$. We say that $\mathbf{u}$ is spectrally stable if the spectrum of $\mathbf{L}(\mathbf{u})$ lies in the left half plane of $\mathbb{C}$, where $\Re(\lambda)\leq0$, and spectrally unstable otherwise. The spectrum of $\mathbf{L}(\mathbf{u})$ is symmetric about the real and imaginary axes because \eqref{eqn2:uv} is in the Hamiltonian form, whereby $\mathbf{u}$ is spectrally stable if and only if the spectrum of $\mathbf{L}(\mathbf{u})$ is contained in the imaginary axis. 

Since $\mathbf{L}(\mathbf{u})$ is $T$ periodic in the $x$ variable, Floquet theory (see \cite[Theorem~2.95]{chicone}, for instance) implies that $\lambda$ is in the spectrum if and only if \eqref{eqn2:spec} has a nontrivial bounded solution such that 
\[
\boldsymbol{\phi}(x+T)=e^{\frac{2\pi ikx}{T}}\boldsymbol{\phi}(x)
\quad\text{for some $k\in\R$}. 
\]
One can write $k=[k]+\tau$, where $[k]\in\mathbb{Z}$, and $\tau\in(-\frac12,\frac12]$ is the Floquet exponent.
We focus our attention to superharmonic and, particularly, short wavelength perturbations, for which $|k|\gg1$. 
Indeed, our numerical experiments (see Figure~\ref{FourthBandAndGapExample}, for instance) show a striking and new instability phenomenon---spectrum off the imaginary axis extending all the way to $\pm i\infty$---for some $\mathbf{u}$, that is, for some $\alpha$, $\beta$, $\gamma$. The aim here is to explain such instability. Subharmonic and, particularly, long wavelength perturbations, for which $|k|\neq 0,\ll1$, will be addressed in \cite{BHW2}.  

Since
\[
\mathbf{L}(\mathbf{u})e^{\frac{2\pi ikx}{T}}=\mqty(O(k) & O(k^{-1}) \\ O(k) & O(k))
\quad\text{as $|k|\to\infty$},
\]
we make the ansatz 
\begin{equation}\label{def2:k>>1}
\lambda=\lambda^{(1)}k+\lambda^{(0)}+\lambda^{(-1)}k^{-1}+\cdots\quad\text{and}\quad
\boldsymbol{\phi}(x)=e^{\frac{2\pi ikx}{T}}
(\boldsymbol{\phi}^{(0)}(x)+\boldsymbol{\phi}^{(-1)}(x)k^{-1}+\boldsymbol{\phi}^{(-2)}(x)k^{-2}+\cdots)
\end{equation}
as $|k|\to\infty$ for some $\lambda^{(1)}$, $\lambda^{(0)}$, $\lambda^{(-1)},\ldots\in\mathbb{C}$ for some $\boldsymbol{\phi}^{(0)}$, $\boldsymbol{\phi}^{(-1)}$, $\boldsymbol{\phi}^{(-2)},\ldots\in L^\infty(\mathbb{R})\times L^\infty(\mathbb{R})$, to be determined. Substituting \eqref{def2:k>>1} into \eqref{eqn2:spec}, 
\begin{align*}
(&\lambda^{(1)}k+\lambda^{(0)}+\lambda^{(-1)}k^{-1}+\cdots)
(\boldsymbol{\phi}^{(0)}+\boldsymbol{\phi}^{(-1)}k^{-1}+\boldsymbol{\phi}^{(-2)}k^{-2}+\cdots)\\
&=\mqty(c\big(\frac{2\pi ik}{T}+\partial_x\big) 
& \big(\frac{2\pi ik}{T}+\partial_x\big)
\big(1-\big(\frac{2\pi ik}{T}+\partial_x\big)^2\big)^{-1}\\
\big(\frac{2\pi ik}{T}+\partial_x\big)(1+2u) 
& c\big(\frac{2\pi ik}{T}+\partial_x\big))
(\boldsymbol{\phi}^{(0)}+\boldsymbol{\phi}^{(-1)}k^{-1}+\boldsymbol{\phi}^{(-2)}k^{-2}+\cdots)\\
&=:(\mathbf{L}^{(1)}k+\mathbf{L}^{(0)}+\mathbf{L}^{(-1)}k^{-1}+\cdots) (\boldsymbol{\phi}^{(0)}+\boldsymbol{\phi}^{(-1)}k^{-1}+\boldsymbol{\phi}^{(-2)}k^{-2}+\cdots)
\end{align*}
as $|k|\to\infty$, where 
\begin{align*}
\mathbf{L}^{(1)}=\mqty(\frac{2\pi ic}{T} & 0 \\ \frac{2\pi i}{T}(1+2u) & \frac{2\pi ic}{T}),\quad
\mathbf{L}^{(0)}=\mqty(c\partial_x & 0 \\ \partial_x(1+2u) & c\partial_x)
\quad\text{and}\quad
\mathbf{L}^{(-1)}=\mqty(0 & -\frac{T}{2\pi i} \\ 0 & 0).  
\end{align*}

At the order of $k$, we gather
\[
\lambda^{(1)}\boldsymbol{\phi}^{(0)}=\mathbf{L}^{(1)}\boldsymbol{\phi}^{(0)}
=\mqty(\frac{2\pi ic}{T} & 0 \\ \frac{2\pi i}{T}(1+2u) & \frac{2\pi ic}{T})\boldsymbol{\phi}^{(0)},
\]
whence
\begin{equation}\label{eqn2:O(k)}
\lambda^{(1)}=\frac{2\pi ic}{T} \quad\text{and}\quad 
\boldsymbol{\phi}^{(0)}=\mqty(0 \\ \phi^{(0)}_{2}),
\end{equation}
where $\phi^{(0)}_{2}$ is bounded and otherwise arbitrary. 

At the order of $1$, we gather
\begin{align*}
\frac{2\pi ic}{T}\mqty( \phi^{(-1)}_1 \\ \phi^{(-1)}_2)+\lambda^{(0)}\mqty(0 \\ \phi^{(0)}_{2})
=&\lambda^{(1)}\boldsymbol{\phi}^{(-1)}+\lambda^{(0)}\boldsymbol{\phi}^{(0)}\\
=&\mathbf{L}^{(1)}\boldsymbol{\phi}^{(-1)}+\mathbf{L}^{(0)}\boldsymbol{\phi}^{(0)} \\
=&\mqty(\frac{2\pi ic}{T} & 0 \\ \frac{2\pi i}{T}(1+2u) & \frac{2\pi ic}{T})\mqty( \phi^{(-1)}_1 \\ \phi^{(-1)}_2)
+\mqty(c\partial_x & 0 \\ \partial_x(1+2u) & c\partial_x)\mqty(0 \\ \phi^{(0)}_{2}),
\end{align*}
whence 
\begin{equation}\label{eqn2:O(1)}
\frac{2\pi i}{T}(1+2u)\phi^{(-1)}_{1}=(\lambda^{(0)}-c\partial_x)\phi^{(0)}_{2}.
\end{equation}
At the order of $k^{-1}$, similarly,
\begin{align*}
\mqty(0 & 0 \\ \frac{2\pi i}{T}(1+2u) & 0)\boldsymbol{\phi}^{(-2)}
=&(\mathbf{L}^{(1)}-\lambda^{(1)})\boldsymbol{\phi}^{(-2)}\\
=&-(\mathbf{L}^{(0)}-\lambda^{(0)})\boldsymbol{\phi}^{(-1)}
-(\mathbf{L}^{(-1)}-\lambda^{(-1)})\boldsymbol{\phi}^{(0)}\\
=&\mqty((\lambda^{(0)}-c\partial_x)\phi^{(-1)}_{1}
+\frac{T}{2\pi i}\phi^{(0)}_{2} \\
-\partial_x((1+2u)\phi^{(-1)}_{1})
+(\lambda^{(0)}-c\partial_x)\phi^{(-1)}_{2}
-\lambda^{(-1)}\phi^{(0)}_{2}),
\end{align*}
which is solvable by the Fredholm alternative, provided that 
\begin{equation}\label{eqn2:O(1/k)}
(\lambda^{(0)}-c\partial_x)\phi^{(-1)}_{1}+\frac{T}{2 \pi i}\phi^{(0)}_{2}=0.
\end{equation}

We can write \eqref{eqn2:O(1)} and \eqref{eqn2:O(1/k)} as
\[
\mqty(\frac{2\pi i}{T}(\lambda^{(0)}-c\partial_x) & 1 \\
-\frac{2\pi i}{T}(1+2u) & (\lambda^{(0)}-c\partial_x))
\mqty(\phi^{(-1)}_{1} \\ \phi^{(0)}_{2})=\boldsymbol{0},
\]
and rewrite more conveniently as
\[
\left(\lambda^{(0)}\boldsymbol{1}-\mqty(c\partial_x & -1 \\ 1+2u & c\partial_x)\right)\mqty(\frac{2\pi i}{T} \phi^{(-1)}_1 \\ \phi^{(0)}_2)=\boldsymbol{0},
\]
where $\boldsymbol{1}$ denotes the identity operator, or as a quadratic pencil as
\begin{equation}\label{eqn2:pencil}
((\lambda^{(0)}-c\partial_x)^2+1+2u)\phi=0.
\end{equation}
Our task is to distinguish whether $\text{Re}(\lambda^{(0)})=0$ or not, for which \eqref{eqn2:pencil} has a nontrivial bounded solution. 
After the change of variables
\begin{equation}\label{def2:y->phi}
\phi=ye^{-\lambda^{(0)}x},
\end{equation}
\eqref{eqn2:pencil} becomes the Hill's differential equation 
\begin{equation}\label{eqn2:Hill}
c^2y''+(1+2u)y=0. 
\end{equation}
Recall~\cite{Eastham,MW} that solutions of \eqref{eqn2:Hill} can be classified as:
\begin{itemize}
\item {\bf Elliptic (band)}. \eqref{eqn2:Hill} has two linearly independent solutions $y_1$ and $y_2$, say, such that $y_1(x+T)=e^{i\sigma T}y_1(x)$ and $y_2(x+T)=e^{-i\sigma T}y_2(x)$ for some $\sigma\in\mathbb{R}$;
\item {\bf Parabolic (band edge)}. \eqref{eqn2:Hill} has two linearly independent solutions $y_1$ and $xy_1+y_2$ such that $y_1(x+T)=y_1(x)$ and $y_2(x+T)=y_2(x)$; and 
\item {\bf Hyperbolic (gap)}. \eqref{eqn2:Hill} has two linearly independent solutions $y_1$ and $y_2$ such that 
\begin{equation}\label{def2:sigma}
y_1(x+T)=e^{\sigma T}y_1(x)\quad \text{and} \quad y_2(x+T) = e^{-\sigma T}y_2(x)\qquad \text{for some $\sigma>0$},
\end{equation}
\end{itemize}
the Lyapunov exponent. A degenerate band edge (closed gap) is included in a band for convenience. 

Therefore, $\lambda^{(0)}$ and bounded solutions of \eqref{eqn2:pencil} can be classified.

\begin{lemma}\label{lem:Hill}
If \eqref{eqn2:Hill} is elliptic (in a band) then \eqref{eqn2:pencil} has a bounded solution if and only if $\Re(\lambda^{(0)})=0$, in which case all solutions are bounded. If \eqref{eqn2:Hill} is parabolic (band edge) then \eqref{eqn2:pencil} has a bounded solution if and only if  $\Re(\lambda^{(0)})=0$, in which case there is only one linearly independent bounded solution.

If \eqref{eqn2:Hill} is hyperbolic (in a gap), on the other hand, then \eqref{eqn2:pencil} has a bounded solution if and only if $\Re(\lambda^{(0)})=\pm\sigma$, where $\sigma>0$ is the Lyapunov exponent of \eqref{eqn2:Hill} (see \eqref{def2:sigma}), in which case there is only one linearly independent bounded solution for each value of $\Re(\lambda^{(0)})$.
\end{lemma}

\begin{proof}
If $\Re(\lambda^{(0)})=0$ then \eqref{def2:y->phi} maps from bounded solutions of \eqref{eqn2:Hill} to bounded solutions of \eqref{eqn2:pencil}. Otherwise, \eqref{def2:y->phi} can map to bounded solutions of \eqref{eqn2:pencil}, provided that $\Re(\lambda^{(0)})=\pm\sigma$. 
\end{proof}

We summarize our conclusion. 

\begin{theorem}\label{thm:Bou}
Let $(\alpha,\beta,\gamma)\in\triangle$ and $\mathbf{u}$ denotes a periodic traveling wave of \eqref{eqn2:uv}, $c$ the wave speed and $T$ the period, depending on $\alpha$, $\beta$, $\gamma$. If \eqref{eqn2:Hill} is elliptic (in a band) then the spectrum of \eqref{eqn2:spec} tends to infinity along the imaginary axis. If \eqref{eqn2:Hill} is hyperbolic (in a gap), 
on the other hand, then the spectrum tends to infinity along $\pm\sigma+i\R$, where $\sigma>0$ is the Lyapunov exponent of \eqref{eqn2:Hill} (see \eqref{def2:sigma}).
\end{theorem}

Consequently, when \eqref{eqn2:Hill} is in a gap, the periodic traveling wave is spectrally unstable to arbitrarily short wavelength perturbations, in marked contrast to modulational instability to arbitrarily long wavelength perturbations. Note that, while we cannot establish stability in this way having \eqref{eqn2:Hill} hyperbolic is a sufficient condition for instability.

We remark that $\Im(\lambda^{(0)})$ does not occupy a crucial role. This is because $\lambda\sim\lambda^{(1)}k=\frac{2\pi ick}{T}$ as $|k|\to\infty$ to leading order (see \eqref{def2:k>>1} and \eqref{eqn2:O(k)}) and $\Im(\lambda^{(0)})$ serves to translate such---otherwise arbitrary---spectrum along the imaginary axis. Indeed, we can take $\Im(\lambda^{(0)})=0$ without loss of generality. 

Moreover, the Lyapnuov exponent of \eqref{eqn2:Hill} can be found in terms of the monodromy matrix. Let $(\alpha,\beta,\gamma)\in\triangle$, and $y_1$ and $y_2$ denote two solutions of \eqref{eqn2:Hill} such that
\[
y_1(0)=1, \quad y_1'(0)=0 \quad \text{and} \quad y_2(0)=0, \quad y_2'(0)=1.
\]
Let
\[
\mathbf{Y}=\mqty(y_1(T) & y_2(T) \\ y_1'(T) & y_2'(T))
\]
be the monodromy matrix of \eqref{eqn2:Hill}, depending on $\alpha$, $\beta$, $\gamma$. Notice that $\det(\mathbf{Y})=1$. 
If \eqref{eqn2:Hill} is in a gap, that is, $|\tr(\mathbf{Y})|>2$, then $\mathbf{Y}$ has an eigenvalue of modulus $>1$, and
\[
\sigma=\frac{\log(|\mu|)}{T},\quad\text{$\mu$ is the eigenvalue of $\mathbf{Y}$ such that $|\mu|>1$}.
\]
Correspondingly, the spectrum of \eqref{eqn2:spec} 
tends to infinity along $\lambda^{(0)}+i\mathbb{R}$, where 
\[
\lambda^{(0)}=\pm\frac{\log(|\mu|)}{T} \quad\Big(\text{mod}\quad\frac{2\pi i c}{T}\Big).
\]
The assumption that $\Re(\lambda^{(0)})\neq 0$ guarantees that $|\mu|\neq 1$, which in turn guarantees that \eqref{def2:y->phi} is valid. That means, \eqref{eqn2:Hill} is not in a band or band edge and, particularly, $\tr(\mathbf{Y})\neq\pm 2$, whereby solutions of \eqref{eqn2:Hill} do not exhibit secular growth.

\subsection{Numerical experiments corroborating analytical predictions}\label{sec2:numerics}

We discuss a Fourier spectral method for computing the spectrum of \eqref{eqn2:spec} numerically. 
For $\tau\in(-\frac12,\frac12]$, the Floquet exponent, let
\[
\boldsymbol{\phi}(x)=e^{\frac{2\pi i\tau x}{T}}\boldsymbol{\psi}(x),  
\quad \boldsymbol{\psi}(x+T)=\boldsymbol{\psi}(x),
\]
whence
\[
\boldsymbol{\psi}(x)=\sum_{k\in\mathbb{Z}}\widehat{\boldsymbol{\psi}}_ke^{\frac{2\pi ikx}{T}},
\quad \text{$\widehat{\boldsymbol{\psi}}_k$ are the Fourier coefficients}.
\]
We rewrite \eqref{eqn2:spec} as
\begin{equation}\label{eqn2:spec'}
\lambda\boldsymbol{\psi}
=\mqty(c\qty(\frac{2\pi i\tau}{T}+\partial_x) & \qty(\frac{2\pi i\tau}{T}+\partial_x)\big(1-\qty(\frac{2\pi i\tau}{T}+\partial_x)^2\big)^{-1} \\ \qty(\frac{2\pi i\tau}{T}+\partial_x)(1+2u) & c\qty(\frac{2\pi i\tau}{T}+\partial_x))\boldsymbol{\psi},
\end{equation}
where  
\[
u(x)=\sum_{k\in\mathbb{Z}}\widehat{u}_ke^{\frac{2\pi ikx}{T}}, 
\quad \text{$\widehat{u}_k$ are the Fourier coefficients},
\]
$c$ and $T$ depend on $(\alpha,\beta,\gamma)\in\triangle$.
We then make the Fourier collocation projection of \eqref{eqn2:spec'} to the subspace of $L^2(-\frac T2,\frac T2)\times L^2(-\frac T2,\frac T2)$, for which $k=-N_k,-N_k+1,\dots,N_k$. 
For each $\tau$, we run the native eigenvalue solver in Mathematica for the resulting $(4N_k+2)\times(4N_k+2)$ matrix, made up of sums and products of diagonal and Toeplitz matrices, representing differential operators and multiplications by periodic functions, respectively. In all of our numerical experiments, $N_k=100$ and we take a $200$-points discretization of $\tau \in (-\frac{1}{2},\frac{1}{2}]$. 

We exploit an analytic formula for the Fourier series of a Jacobi elliptic function \cite{Kiper} 
\[
\sn^2(x,\sqrt{m}) = \frac{K(\sqrt{m})-E(\sqrt{m})}{m K(\sqrt{m})} - \frac{2\pi^2}{mK^2(\sqrt{m})} \sum_{k=1}^\infty\frac{k q^k}{1-q^{2k}} \cos(\frac{\pi kx}{K(\sqrt{m})}),
\]
to evaluate $\widehat{u}_k$ numerically, where $K(\sqrt{m})$ and $E(\sqrt{m})$ denote the complete elliptic integrals of the first and second kinds, $m\in(0,1)$ is the elliptic parameter, and $q=e^{-\frac{\pi K(\sqrt{1-m})}{K(\sqrt{m})}}.$ 
As an additional numerical check we integrate \eqref{eqn2:u0} (or \eqref{eqn2:u0'}) numerically and approximate the Fourier coefficients by numerical quadrature. The results are in excellent agreement. 

When \eqref{eqn2:Hill} is in a gap, we also integrate \eqref{eqn2:Hill} numerically to find the eigenvalues of the monodromy matrix. They determine the spectrum of \eqref{eqn2:spec} off the imaginary axis at infinity. 

\begin{figure}[htbp]
\begin{center}
\includegraphics[width=0.495\textwidth]{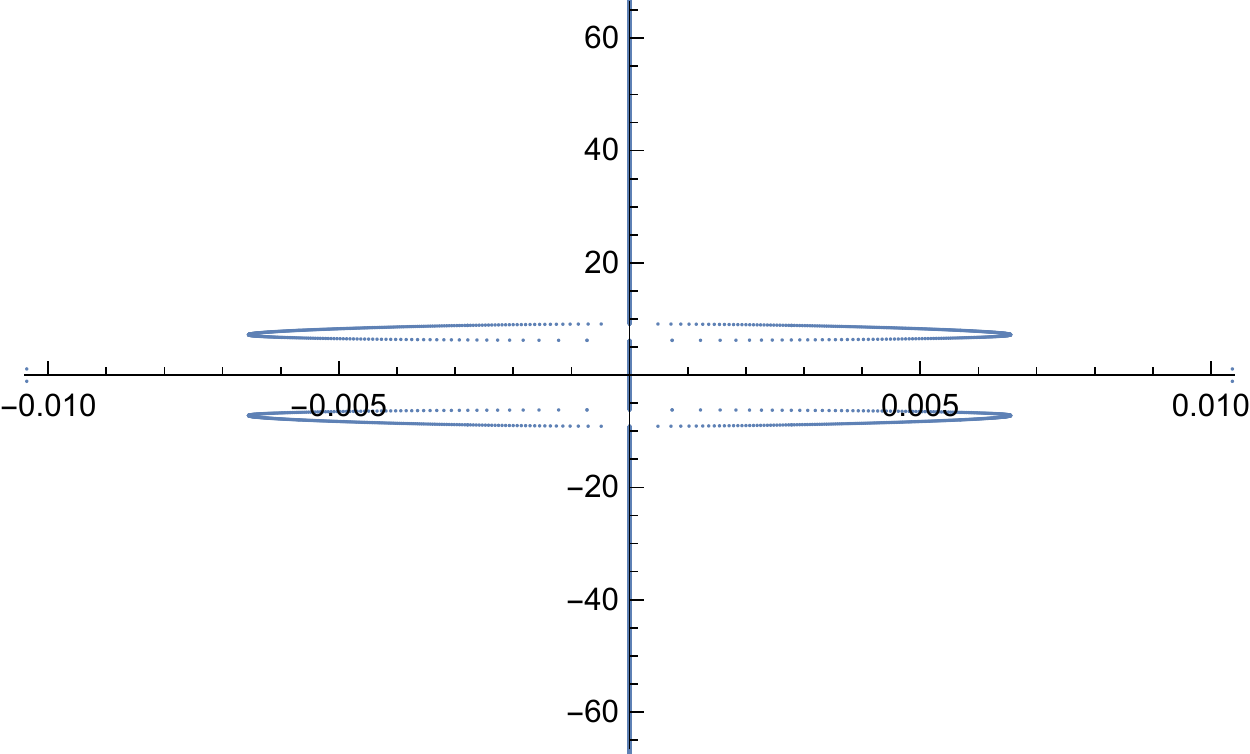}
\includegraphics[width=0.495\textwidth]{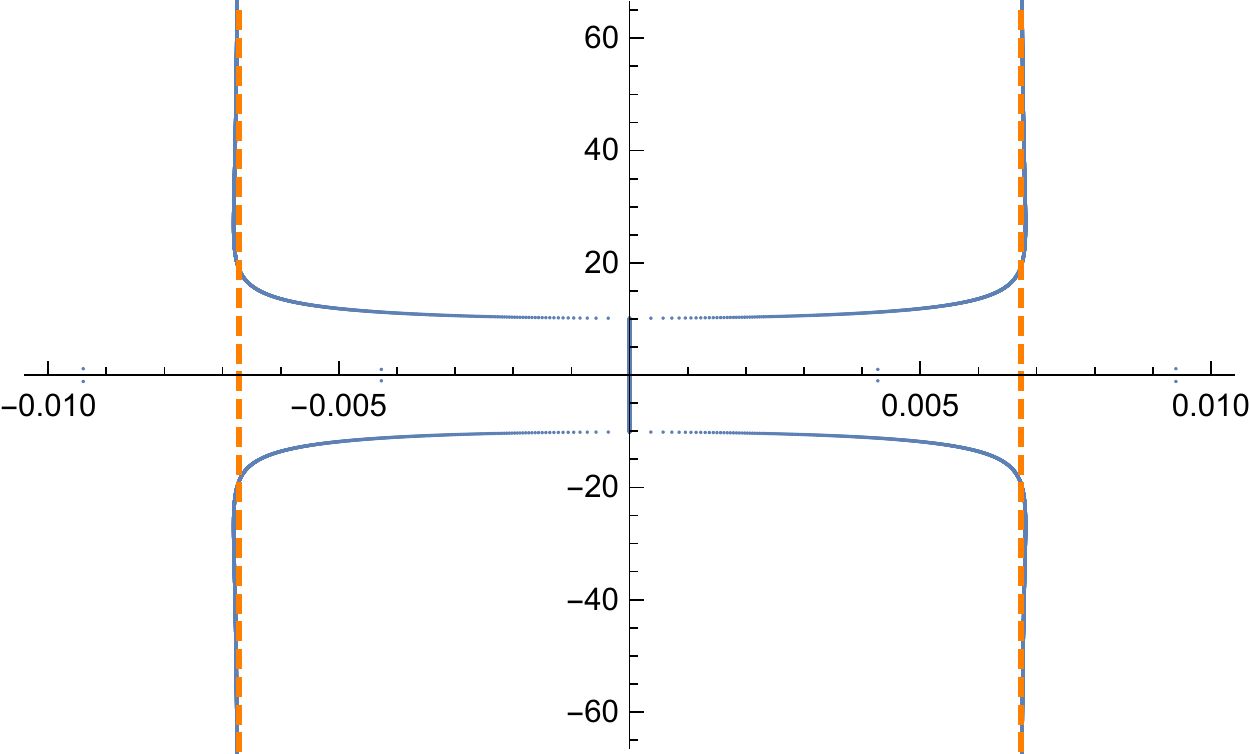}
\caption{The numerically computed spectrum of \eqref{eqn2:spec}, where $u$, $c$, $T$ depend on $\alpha$,~$\beta$,~$\gamma$. On the left, $\alpha=-0.7600$, $\beta=-0.006403$, $\gamma=1$, for which \eqref{eqn2:Hill} is in a band. Notice modulational instability near $0\in\mathbb{C}$, although it is hard to distinguish, and finite wavelength instability near $\pm 7.5i$. The spectrum lies along the imaginary axis otherwise. 
On the right, $\alpha=-0.7872$, $\beta=-0.006403$, $\gamma=1$, for which \eqref{eqn2:Hill} is in a gap. The spectrum tends towards infinity along the dashed lines $\pm\sigma+i\mathbb{R}$, where $\sigma\approx0.006726$. We report modulational instability near $0\in\mathbb{C}$.}
\label{FourthBandAndGapExample}
\end{center}
\end{figure}

Figure~\ref{FourthBandAndGapExample} provides two examples of the numerically computed spectrum of \eqref{eqn2:spec}. The periodic traveling waves are close together in $\triangle$ (see also Figure~\ref{BandsAndGaps})---indeed, only the value of $\alpha$ differs slightly---and yet they exhibit a dramatic difference in the spectrum at $\pm i\infty$. In the left panel, for which \eqref{eqn2:Hill} is elliptic (in a band), the spectrum lies apparently along the imaginary axis outside of some bounded set. In the right panel, on the other hand, the spectrum tends towards infinity along $\pm\sigma+i\mathbb{R}$, where $\sigma\approx0.006726$, and $\sigma$ agrees well with $\frac{\log(|\mu|)}{T}$, where $\mu$ is numerically computed eigenvalue of the monodromy matrix of \eqref{eqn2:Hill} such that $|\mu|>1$. Therefore the numerical result corroborates Theorem~\ref{thm:Bou}. There is a modulational instability to long wavelength perturbations near $0\in\mathbb{C}$ in both panels, although it is hard to distinguish at this scale. Also there is spectral instability to finite wavelength perturbations away from $0\in\mathbb{C}$ in the left panel. 

\begin{figure}[htbp]
\begin{center}
    \includegraphics[width=0.495\textwidth]{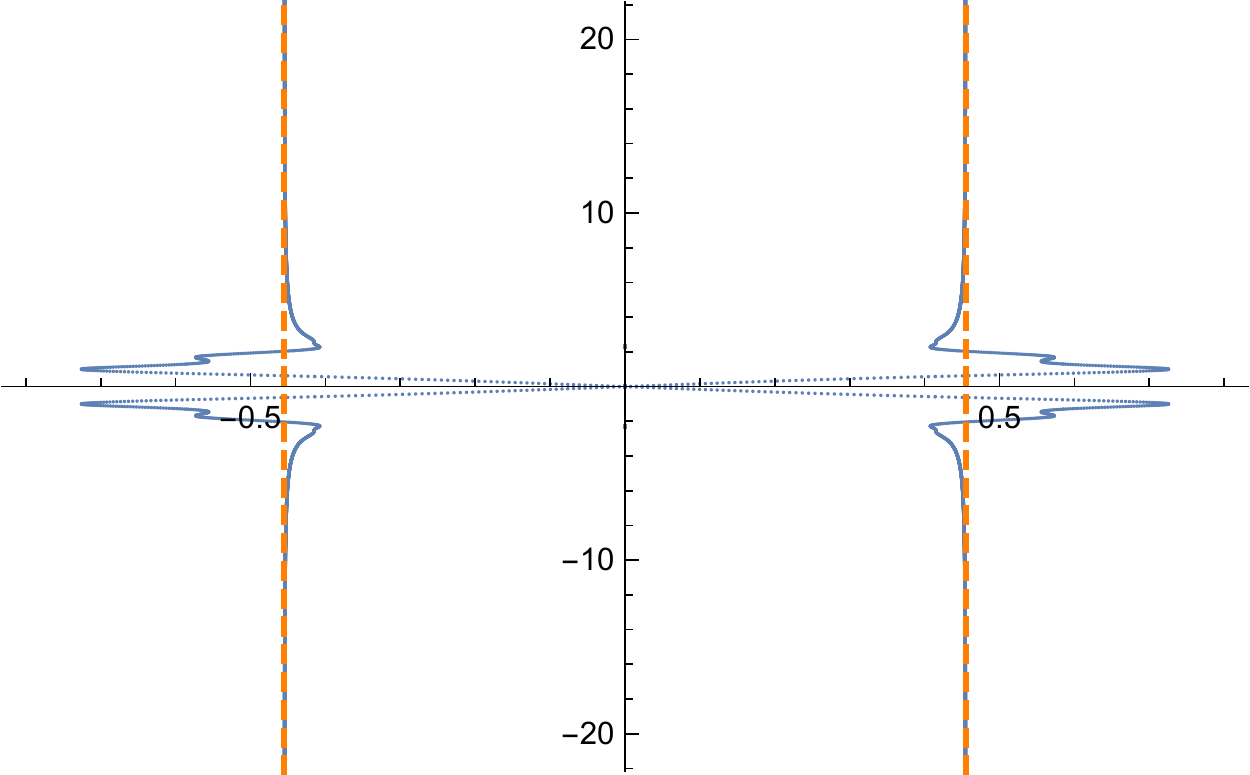}
    \includegraphics[width=0.495\textwidth]{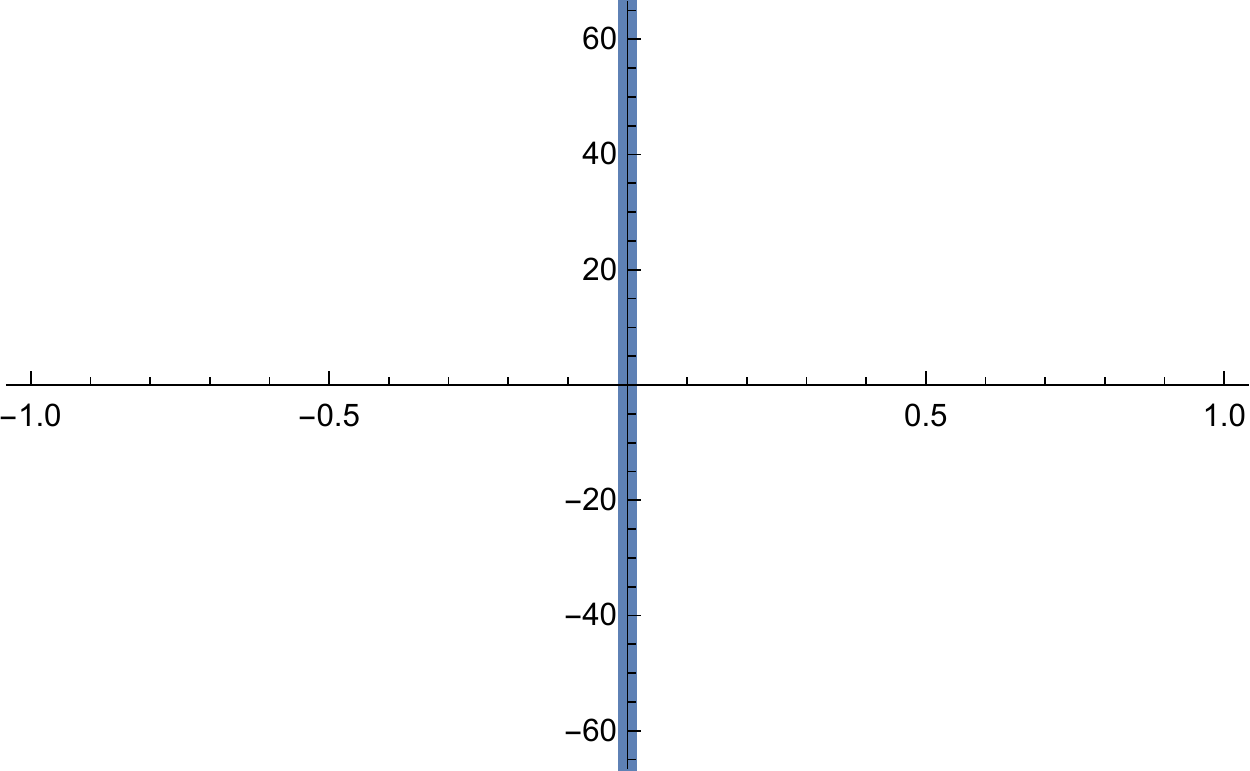}
\caption{The spectrum of \eqref{eqn2:spec}. On the left, $\alpha=-1.246$, $\beta=-1.149$, $\gamma=1$, for which \eqref{eqn2:Hill} is in a gap. The spectrum tends towards infinity along the dashed lines $\pm\sigma+i\mathbb{R}$, where $\sigma\approx0.455$. Notice modulational instability near $0\in\mathbb{C}$ (see also Figure~\ref{SecondGapExampleZoomedIn}). On the right, $\alpha=-2.034$, $\beta=0.7131$, $\gamma=1$, for which \eqref{eqn2:Hill} is in a band. The spectrum lies along the imaginary axis, suggesting spectral stability.}
\label{SecondGapExample}
\end{center}
\end{figure}

\begin{figure}[htbp]
\begin{center}
    \includegraphics[width=0.5\textwidth]{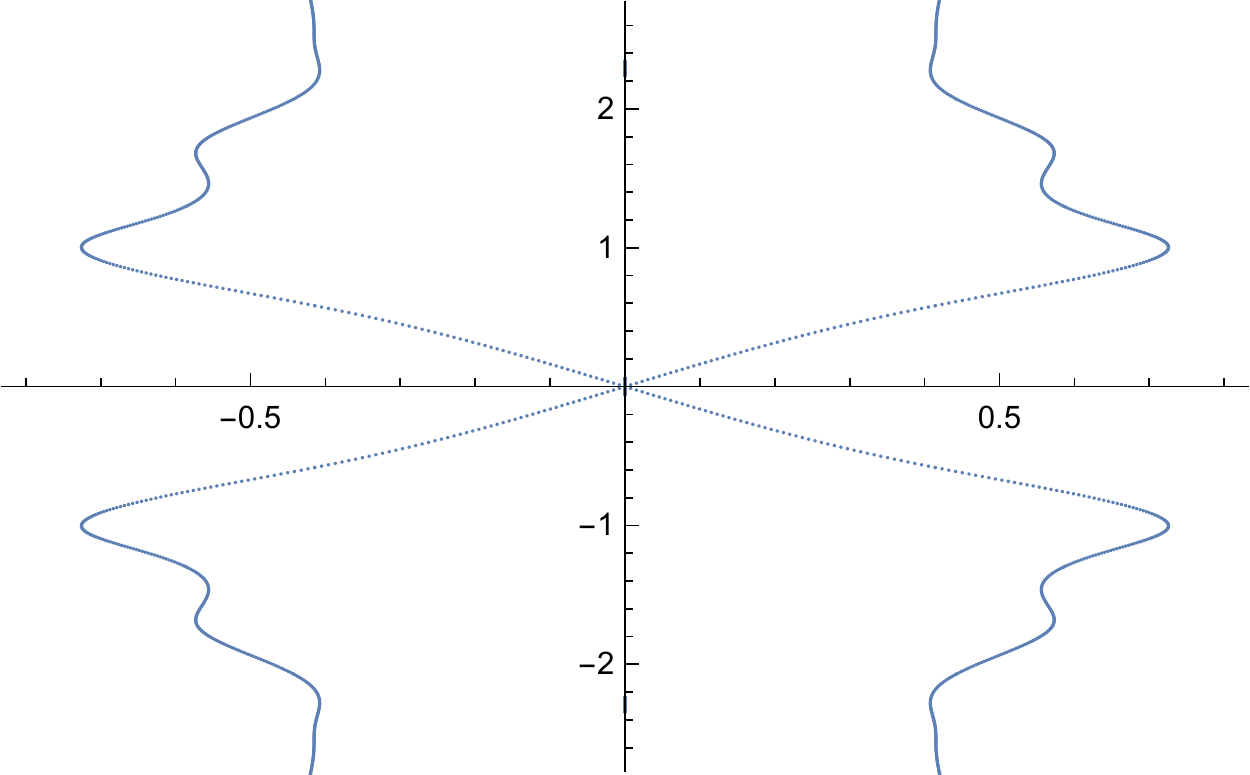}
\caption{$\alpha=-1.246$, $\beta=-1.149$, $\gamma=1$ (see the left panel of Figure~\ref{SecondGapExample}), and the spectrum of \eqref{eqn2:spec} near $0\in\mathbb{C}$ for modulational instability.}
\label{SecondGapExampleZoomedIn}
\end{center}
\end{figure}

In the left panel of Figure~\ref{SecondGapExample}, for which \eqref{eqn2:Hill} is in a gap---the second gap of the associated Lam\'e equation (see Section~\ref{sec2:Lame})---the numerically computed spectrum tends towards infinity along $\pm\sigma+i\mathbb{R}$, where $\sigma\approx0.455$, far greater than the right panel of Figure~\ref{FourthBandAndGapExample}, and $\sigma$ agrees well with $\frac{\log(|\mu|)}{T}$, where $\mu$ is the numerically computed eigenvalue of the monodromy matrix of \eqref{eqn2:Hill} such that $|\mu|>1$. See also Figure~\ref{SecondGapExampleZoomedIn} for modulational instability.  
 
Last but not least, in the right panel of Figure~\ref{SecondGapExample}, for which \eqref{eqn2:Hill} is in a band---indeed, \eqref{eqn2:Lame} is in the third band (see Section~\ref{sec2:Lame})---the numerical result is consistent with  Theorem~\ref{thm:Bou} as the spectrum tends towards infinity along the imaginary axis. The greatest real part of the numerically computed spectrum is of the order of $10^{-8}$, suggesting spectral stability. 


\subsection{Analytic formulae classifying the spectrum at infinity}\label{sec2:Lame}


We can rewrite \eqref{def2:sn} as 
\[
u(x)=\gamma-6ma^2c^2\sn^2(ax,\sqrt{m}),
\]
where $m$ and $a$ are in \eqref{def2:am}, and $c$ is in the third equation of \eqref{eqn2:cbE->abc}, 
whence \eqref{eqn2:Hill} as
\[
y''+\qty(\frac{2\gamma+1}{c^2}-12ma^2\sn^2(ax,\sqrt{m}))y=0.
\]
After the change of variables $x \mapsto \frac{x}{a}$, we arrive at the Lam\'e equation (in the Jacobi form) \cite{Arscott}
\begin{subequations}\label{eqn2:Lame}
\begin{equation}\label{eqn2:ell}
y''+(\ell-3(3+1)m\sn^2(x,\sqrt{m}))y=0,
\end{equation}
where 
\begin{equation}\label{def2:ell}
\ell=\frac{6(2\gamma+1)}{\gamma-\alpha}=4+4m+\frac{1}{a^2}
\end{equation}
\end{subequations}
by \eqref{eqn2:cbE->abc} and \eqref{def2:am}.

Recall~\cite{Arscott,KS} that \eqref{eqn2:ell} has exactly three finite (open) gaps plus one semi-infinite gap, whence four disjoint bands, whose band edges can be found in closed form in terms of the roots of some polynomials. Specifically, let $\{\ell^P_j\}_{j=1}^\infty$ denote the periodic eigenvalues of \eqref{eqn2:ell} and $\{\ell^A_j\}_{j=1}^\infty$ the anti-periodic eigenvalues, respectively, such that 
\[
-\infty<\ell_1^P<\ell_1^A<\ell_2^A<\ell_2^P<\ell_3^P<\ell_3^A<\ell_4^A\leq\cdots,
\]
and
\begin{align*}
\ell_1^P &= 2+5m-2\theta_1, 
&\ell_2^P &= 4+4m, 
&\ell_3^P &= 2+5m+2\theta_1, & \\
\ell_1^A &= 5+2m-\theta_2,
&\ell_2^A &= 5+5m-2\theta_3, 
&\ell_3^A &= 5+2m+2\theta_2,
&\ell_4^A& = 5+5m+2\theta_3,
\end{align*}
where
\[
\theta_1 = \sqrt{1-m+4m^2}, \quad
\theta_2 = \sqrt{4-m+m^2},\quad
\theta_3 = \sqrt{4-7m+4m^2},
\]
and $\ell_j^A=\ell_j^P=\ell_{j+1}^P=\ell_{j+1}^A$ for $j\geq 4,\in\mathbb{Z}$. 
The bands of \eqref{eqn2:ell} consist of
\begin{equation}\label{def2:bands}
(\ell_1^P,\ell_1^A), \quad (\ell_2^A,\ell_2^P), \quad (\ell_3^P,\ell_3^A)\quad\text{and}\quad (\ell_4^A,\infty),
\end{equation}
and the gaps consist of 
\begin{equation}\label{def2:gaps}
(-\infty,\ell_1^P),\quad (\ell_1^A,\ell_2^A), \quad (\ell_2^P,\ell_3^P) \quad\text{and}\quad (\ell_3^A,\ell_4^A).
\end{equation}
We refer to $(-\infty,\ell_1^P)$ as the zeroth gap for convenience. 

If \eqref{eqn2:Lame} lies in a band, that is, \eqref{def2:ell} is in \eqref{def2:bands}, then the spectrum of \eqref{eqn2:spec} tends to infinity along the imaginary axis. If \eqref{eqn2:Lame} is in a gap, that is, \eqref{def2:gaps}, on the other hand, then the spectrum tends to infinity along $\pm\sigma+i\mathbb{R}$ for some $\sigma>0$. Therefore we achieve analytic formulae, depending on $m$ and $a$, whence depending on $\alpha$, $\beta$, $\gamma$ by \eqref{def2:am}, which classify the spectrum of \eqref{eqn2:spec} at infinity. 

\begin{figure}[htbp]
\begin{center}
\includegraphics[width=.495\textwidth]{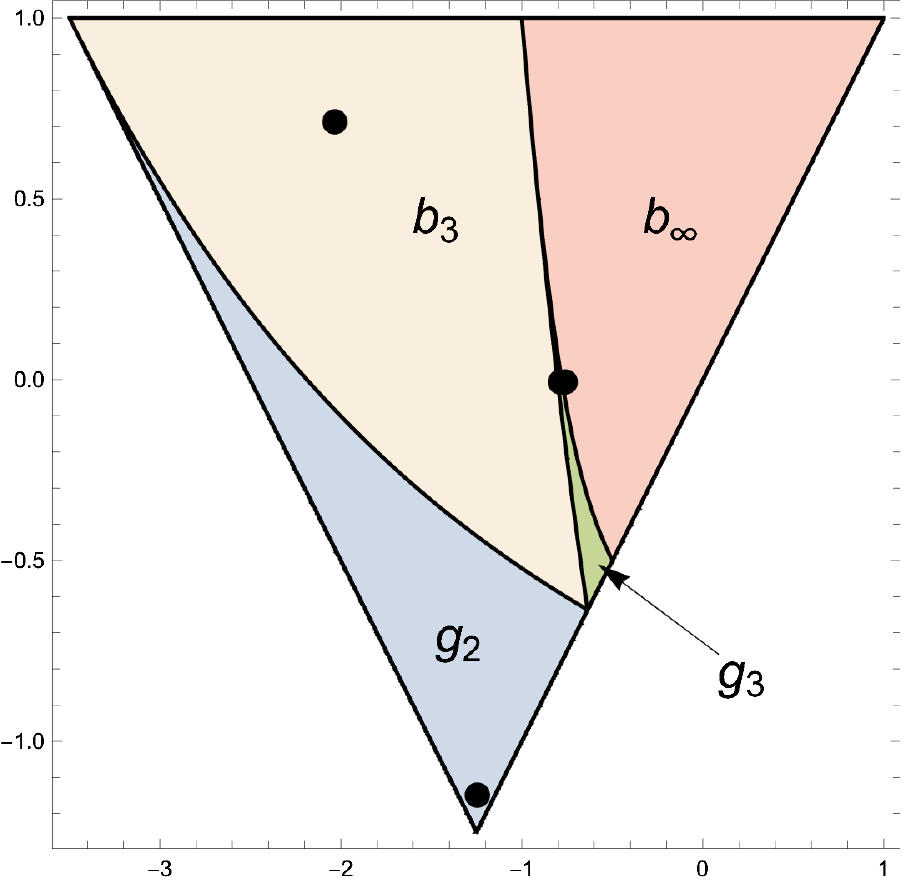}
\includegraphics[width=.495\textwidth]{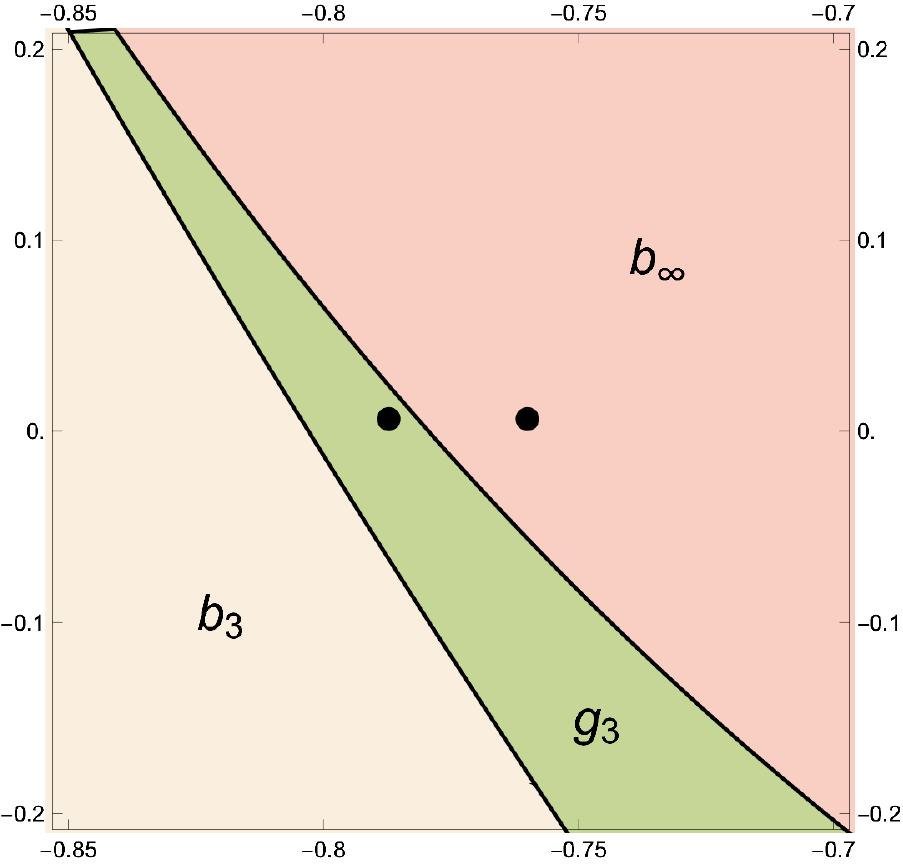}
\caption{$\triangle$ in the $(\alpha,\beta)$ plane when $\gamma=1$: $g_2$ denotes the region, for which \eqref{eqn2:Lame} is in the second gap, that is, \eqref{def2:ell} is in $(\ell_2^P,\ell_3^P)$, and $g_3$ the third gap $(\ell_3^A,\ell_4^A)$. Also $b_3$ denotes the region, for which \eqref{eqn2:Lame} is in the third band, that is, $(\ell_3^P,\ell_3^A)$, and $b_\infty$ the band $(\ell_4^A,\infty)$. 
The bullet points correspond to the values of $(\alpha,\beta)$ for Figures~\ref{FourthBandAndGapExample} and \ref{SecondGapExample}. On the right is a close-up of the values of $(\alpha,\beta)$ for Figure~\ref{FourthBandAndGapExample}. }
\label{BandsAndGaps}
\end{center}
\end{figure}

Figure~\ref{BandsAndGaps} shows bands and gaps of \eqref{eqn2:Lame} in the $(\alpha,\beta)$ plane when $\gamma=1$. Notice that not all bands and gaps are present in $\triangle$. Specifically, the zeroth and first gaps, and the first and second bands are not present. But when a band or a gap is present in $\triangle$, it is entirely contained in $\triangle$. For a proof, we observe that: 
\begin{itemize}
\item The top edge of $\triangle$, where $\beta=1$, corresponds to $m=0$, and the right edge, where $\alpha=\beta$, corresponds to $m=1$. More generally, the line segment $\frac{1-\beta}{1-\alpha}=m$ passes through $(\alpha,\beta)=(1,1)$, whose slope is $m$.
\item The left edge of $\triangle$, where $\alpha+\beta=-\frac52$, corresponds to $\frac{18}{1-\alpha}=\ell_2^P$.
\item $\frac{18}{1-\alpha}$ can be made as large as desired for $\alpha$ sufficiently close to $1$ such that $\alpha<\beta<1$, that is, in the upper right corner of $\triangle$.
\end{itemize}

More generally, for any $m\in(0,1)$ and for any $\ell\in(\ell_2^P,\infty)$, recalling $\ell_2^P=4+4m$, there exists $(\alpha,\beta,\gamma)\in\triangle$ such that
\[
\frac{\gamma-\beta}{\gamma-\alpha}=m \quad\text{and}\quad \frac{6(2\gamma+1)}{\gamma-\alpha}=\ell.
\]

We summarize our conclusion.

\begin{corollary}
Let $(\alpha,\beta,\gamma)\in\Delta$ and $\mathbf{u}$ denotes a periodic traveling wave of \eqref{eqn2:uv}, $c$ the wave speed and $T$ the period, depending on $m$ and $a$, whence depending on $\alpha$, $\beta$, $\gamma$. If 
\begin{gather*}
\frac{1}{a^2}<m-2+2\sqrt{1-m+4m^2}
\intertext{or}
1-2m-2\sqrt{4-m+m^2} <\frac{1}{a^2}<1+m-2\sqrt{4-7m+4m^2}
\end{gather*}
then the spectrum of \eqref{eqn2:spec} tends to infinity along $\pm\sigma+i\mathbb{R}$ for some $\sigma>0$.
\end{corollary}

The bullet points in the regions $g_2$ and $g_3$ of Figure~\ref{BandsAndGaps}, for which \eqref{eqn2:Lame} is in the second and third gaps, correspond to the values of $\alpha$, $\beta$, $\gamma$ for the left panel of Figure~\ref{SecondGapExample} and the right panel of Figure~\ref{FourthBandAndGapExample}, respectively. The spectrum in each of the panels tends towards infinity along $\pm\sigma+i\mathbb{R}$ for some $\sigma>0$. 
The bullet points in the regions $b_3$ and $b_\infty$, for which \eqref{eqn2:Lame} is in the bands $(\ell^P_3,\ell^A_3)$ and $(\ell^A_4,\infty)$, on the other hand, correspond to the values of $\alpha$, $\beta$, $\gamma$ for the right panel of Figure~\ref{SecondGapExample} and the left panel of Figure~\ref{FourthBandAndGapExample}. The spectrum tends towards infinity along the imaginary axis. 

\section{The Benney--Luke equation}\label{sec:BL}

We turn our attention to \eqref{eqn1:BL}, where $a=0$ and, for simplicity of notation, $b=1$, that is,
\begin{equation}\label{eqn:BL}
u_{tt}-u_{xx}-u_{xxtt}+u_tu_{xx}+2u_xu_{xt}=0,
\end{equation}
whose dispersion relation
\[
\omega^2(k)=\frac{k^2}{1+k^2}
\]
remains bounded for all $k\in\mathbb{R}$. Proceeding as in \cite{AnguloQuintero2007,PegoQuintero1999}, let 
\begin{equation}\label{def3:vw}
v=u_x\quad\text{and}\quad w=(1-\partial_{x}^2)u_t+\frac12v^2,
\end{equation}
and we rewrite \eqref{eqn:BL} as 
\begin{equation}\label{eqn3:vw}
\mqty(v \\ w)_t=\mqty( (1-\partial_x^2)^{-1}\qty(w-\frac12v^2)\\ 
v-v(1-\partial_x^2)^{-1}\qty(w-\frac12v^2))_x.
\end{equation}
Throughout the section, $\mathbf{v}=\mqty(v\\w)$. We remark that \eqref{eqn3:vw} is in the  Hamiltonian form, for which 
\[
H(\mathbf{v})=\frac{1}{2}\int \left(v^2+\qty(w-\frac{1}{2}v^2)(1-\partial_x^2)^{-1}\qty(w-\frac{1}{2}v^2)\right)~\dd{x}
\]
is the Hamiltonian. In addition to $H$, \eqref{eqn3:vw} has three conserved quantities 
\[
P(\mathbf{v})=\int vw~\dd{x}, \quad
M_1(\mathbf{v})=\int v~\dd{x} \quad \text{and} \quad M_2(\mathbf{v})=\int w~\dd{x}. 
\]

\subsection{Parametrization of periodic traveling waves}\label{sec3:periodic}

Similarly as in Section~\ref{sec2:periodic}, a traveling wave of \eqref{eqn3:vw} takes the form $\mathbf{v}(x-ct-x_0)$ for some $c\neq0,\in\mathbb{R}$ for some $x_0\in\mathbb{R}$, and it satisfies 
\[
\delta(H+cP+b_1M_1+b_2M_2)({\mathbf v})=\mathbf{0} 
\]
for some $b_1,b_2\in\mathbb{R}$. That is,
\begin{equation}\label{eqn3:vw0}
\begin{aligned}
&v-v(1-\partial_x^2)^{-1}\qty(w-\frac12v^2)+cw+b_1=0, \\
&(1-\partial_x^2)^{-1}\qty(w-\frac12v^2)+cv+b_2=0. 
\end{aligned}
\end{equation}
We restrict our attention to periodic solutions of \eqref{eqn3:vw0}. The first equation of \eqref{def3:vw} implies that $v$ has mean zero over one period. Also the second equation of \eqref{def3:vw} implies that $w-\frac12v^2=-c(1-\partial_x^2)u_x$ has mean zero over the period. Therefore the second equation of \eqref{eqn3:vw0} dictates that $b_2=0$. 
In what follows, we drop the subscript and refer to $b_1$ as $b$ for simplicity of notation. Eliminating $w$ from \eqref{eqn3:vw0}, we arrive at
\begin{equation}\label{eqn3:v0}
c^2v''+\frac32cv^2+(1-c^2)v+b=0.    
\end{equation}
Multiplying \eqref{eqn3:v0} through by $v'$ and integrating, moreover,
\begin{equation}\label{eqn3:E-V}
\frac12c^2(v')^2=E-\frac12cv^3-\frac12(1-c^2)v^2-bv=:E-V(v;c,b)
\end{equation}
for some $E\in\mathbb{R}$. Since \eqref{eqn3:E-V} remains invariant under 
\[
v\mapsto-v \quad\text{and}\quad (c,b,E)\mapsto(-c,-b,E),
\]
we can take $c>0$ without loss of generality. Since \eqref{eqn3:E-V} remains invariant under the translation of the $x$ axis, we can mod out $x_0$. 

One can work out the existence of non-constant periodic solutions of \eqref{eqn3:E-V} and, hence, non-constant periodic traveling waves of \eqref{eqn3:vw}, depending on $c$, $b$, $E$,  or depending on the roots of the cubic polynomial $E-V(v;c,b)$. But it is inconvenient to impose that $v$ has mean zero over the period when one parametrizes the solutions by the roots of $E-V(v;c,b)$. Instead we restrict our attention to $c>0$ and give periodic and mean-zero solutions of \eqref{eqn3:v0} in closed form in terms of the Jacobi elliptic functions: 
\begin{equation}\label{def3:sn}
v(x)=\frac{4ma^2}{\sqrt{1+4(1+m-3mM(m))a^2}}(\sn^2(ax,\sqrt{m})-M(m)),
\end{equation}
where
\begin{equation}\label{def3:M}
M(m) \vcentcolon =\frac{K(\sqrt{m})-E(\sqrt{m})}{mK(\sqrt{m})}.
\end{equation}
Here $m\in(0,1)$ is the elliptic parameter, and
\[
K(\sqrt{m})=\int_0^1\frac{\dd{s}}{\sqrt{(1-s^2)(1-ms^2)}}\quad\text{and}\quad
E(\sqrt{m})=\int_0^1\sqrt{\frac{1-ms^2}{1-s^2}}~\dd{s}   
\]
are the complete elliptic integrals of the first and second kinds (not to be confused with $E$ in \eqref{eqn3:E-V}). Particularly, \eqref{def3:sn} has mean zero over the period 
\begin{equation}\label{def3:T}
T=\frac{2K(\sqrt{m})}{a}.
\end{equation}
A straightforward calculation reveals that 
\begin{align}
c=&\frac1{\sqrt{1+4(1+m-3mM(m))a^2}}~(>0),\label{def3:c}\\
b=&\frac{8(1 - 2(1+m)M(m) + 3 m M^2(m))ma^4}{\sqrt{1+4(1+m-3mM(m))a^2}^3}, \notag \\
E=&\frac{32(1-M(m))(1-mM(m))m^2M(m)a^6}{(1+4(1+m-3mM(m))a^2)^2}.\notag 
\end{align}
We take $a>0$ without loss of generality. If $m\leq m_0$, where 
\begin{equation}\label{def3:m0}
\text{$m_0\approx0.961$ is the unique root of $1+m-3mM(m)$},
\end{equation}
then \eqref{def3:sn} is defined for all $a$. If $m>m_0$, on the other hand, then $a<\frac{1}{2\sqrt{3mM(m)-m-1}}$ must hold true. 

To summarize, whenever 
\begin{equation}\label{def3:square}
(m,a)\in\square:=\bigg\{(m,a)\in (0,1)\times (0,\infty): 
a<\frac{1}{2\sqrt{3mM(m)-m-1}}\quad\text{for $m<m_0$} \bigg\},
\end{equation}
where $m_0$ is in \eqref{def3:m0}, \eqref{def3:sn} and \eqref{def3:M} give a periodic traveling wave of \eqref{eqn3:vw}, depending on $m$ and $a$, and $v$ has mean zero over the period.  
Figure~\ref{fig:BLBand_Gap} shows $\square$ in the $(m,a)$ plane.

\subsection{Asymptotic spectral analysis to short wavelength perturbations}\label{sec3:spec}

Let $(m,a)\in\square$ (see \eqref{def3:square}) and $\mathbf{v}$ denotes a periodic traveling wave of \eqref{eqn3:vw} (see \eqref{eqn3:vw0}, \eqref{def3:sn}, \eqref{def3:M}), $c$ the wave speed (see \eqref{def3:c}) and $T$ the period (see \eqref{def3:T}), depending on $m$ and $a$. Linearizing \eqref{eqn3:vw} about $\mathbf{v}$ in the moving frame of reference, and seeking a solution of the form $e^{\lambda t}\boldsymbol{\phi}(x)$, say, where $\lambda\in\mathbb{C}$, we arrive at 
\begin{equation}\label{eqn3:spec}
\lambda\boldsymbol{\phi}=
\mqty(\partial_x(c-(1-\partial_x^2)^{-1}v) & \partial_x(1-\partial_x^2)^{-1} \\
\partial_x(1+cv+v(1-\partial_x^2)^{-1}v) & \partial_x(c-(1-\partial_x^2)^{-1}v) )\boldsymbol{\phi}
=:\mathbf{L}(\mathbf{v})\boldsymbol{\phi}.
\end{equation}
Similarly as in Section~\ref{sec2:spec}, $\mathbf{v}$ is spectrally stable if and only if the spectrum of $\mathbf{L}(\mathbf{v}): H^1(\mathbb{R})\times H^1(\mathbb{R})\subset L^2(\mathbb{R})\times L^2(\mathbb{R})\to L^2(\mathbb{R})\times L^2(\mathbb{R})$ is contained in the imaginary axis. Similarly as in Section~\ref{sec2:spec}, $\lambda$ is in the spectrum of $\mathbf{L}(\mathbf{v})$ if and only if \eqref{eqn3:spec} has a nontrivial bounded solution such that 
\[
\boldsymbol{\phi}(x+T)=e^{\frac{2\pi ikx}{T}}\boldsymbol{\phi}(x) \quad\text{for some $k\in\mathbb{R}$}.
\]
We focus our attention to $|k|\gg1$.

Similarly as in Section~\ref{sec2:spec}, let
\[
\lambda=\lambda^{(1)}k+\lambda^{(0)}+\lambda^{(-1)}k^{-1}+\cdots\quad\text{and}\quad
\boldsymbol{\phi}(x)=e^{\frac{2\pi ikx}{T}}
(\bm{\phi}^{(0)}(x)+\bm{\phi}^{(-1)}(x)k^{-1}+\bm{\phi}^{(-2)}(x)k^{-2}+\cdots)
\]
as $|k|\to\infty$ for some $\lambda^{(1)}$, $\lambda^{(0)}$, $\lambda^{(-1)},\ldots\in\mathbb{C}$ for some $\boldsymbol{\phi}^{(0)}$, $\boldsymbol{\phi}^{(-1)}$, $\boldsymbol{\phi}^{(-2)},\ldots\in L^\infty(\mathbb{R})\times L^\infty(\mathbb{R})$, so that \eqref{eqn3:spec} becomes
\begin{multline*}
(\lambda^{(1)}k+\lambda^{(0)}+\lambda^{(-1)}k^{-1}+\cdots)
(\bm{\phi}^{(0)}+\bm{\phi}^{(-1)}k^{-1}+\bm{\phi}^{(-2)}k^{-2}+\cdots)\\
=\mqty((\frac{2\pi ik}{T}+\partial_x)(c-(1-\qty(\frac{2\pi i k}{T}+\partial_x)^2)^{-1} v) & 
(\frac{2\pi i k}{T}+\partial_x)(1-\qty(\frac{2\pi i k}{T}+\partial_x)^2)^{-1} \\
(\frac{2\pi i k}{T}+\partial_x)(1+c v+v(1-\qty(\frac{2\pi i k}{T}+\partial_x)^2)^{-1}v) & 
(\frac{2\pi i k}{T}+\partial_x)(c-(1-\qty(\frac{2\pi i k}{T}+\partial_x)^2)^{-1}v))\\
\cdot(\bm{\phi}^{(0)}+\bm{\phi}^{(-1)}k^{-1}+\bm{\phi}^{(-2)}k^{-2}+\cdots)
\end{multline*}
as $|k|\to\infty$, where
\[
\bigg(1-\qty(\frac{2\pi ik}{T}+\partial_x)^2\bigg)^{-1}=\frac{T^2}{4\pi^2k^2}+O(k^{-3})\quad \text{as $|k|\to\infty$}.
\]

At the order of $k$, 
we gather
\[
\lambda^{(1)}\boldsymbol{\phi}^{(0)}=
\mqty(\frac{2\pi ic}{T} & 0 \\ \frac{2\pi i}{T}(1+cv) & \frac{2\pi ic}{T})\boldsymbol{\phi}^{(0)},
\]
whence 
\[
\lambda^{(1)}=\frac{2\pi i c}{T} \quad\text{and}\quad \boldsymbol{\phi}^{(0)}=\mqty(0 \\ \phi_{2}^{(0)}),
\]
where $\phi_2^{(0)}$ is bounded and otherwise arbitrary. 
At the order of $1$, we gather
\[
\lambda^{(1)} \bm{\phi}^{(-1)}+\lambda^{(0)} \bm{\phi}^{(0)}=
\mqty(\frac{2\pi i c}{T} & 0 \\ \frac{2\pi i}{T}(1+c v) & \frac{2\pi i c}{T} )\bm{\phi}^{(-1)}+\mqty(c\partial_x & 0 \\ \partial_x(1+c v) & c\partial_x)\bm{\phi}^{(0)} 
\]
or, equivalently, 
\[
\mqty(0 & 0 \\ \frac{2\pi i}{T}(1+cv) & 0)\bm{\phi}^{(-1)}
=\mqty(0 \\ (\lambda^{(0)}-c\partial_x)\phi^{(0)}_{2}),
\]
whence
\[
\frac{2\pi i}{T}(1+cv)\phi_1^{(-1)}=(\lambda^{(0)}-c\partial_x)\phi_2^{(0)}.
\]
At the order of $k^{-1}$, similarly,
\begin{multline*}
\lambda^{(1)}\bm{\phi}^{(-2)}+\lambda^{(0)}\bm{\phi}^{(-1)}+\lambda^{(-1)}\bm{\phi}^{(0)} \\
=\mqty(\frac{2\pi i c}{T} & 0 \\ \frac{2\pi i}{T}(1+cv) & \frac{2\pi i c}{T} )\bm{\phi}^{(-2)}
+\mqty(c\partial_x & 0 \\ \partial_x(1+cv) & c\partial_x)\bm{\phi}^{(-1)}
+\mqty(-\frac{T}{2\pi i} v & \frac{T}{2\pi i} \\ \frac{T}{2\pi i} v^2 & -\frac{T}{2\pi i} v)\bm{\phi}^{(0)}
\end{multline*}
or, equivalently, 
\[
\mqty(0 & 0 \\ \frac{2\pi i}{T}(1+cv) & 0)\bm{\phi}^{(-2)} =  \mqty((\lambda^{(0)}-c\partial_x)\phi^{(-1)}_1-\frac{T}{2\pi i}\phi^{(0)}_2 \\
-\partial_x((1+cv)\phi_{1}^{(-1)})+(\lambda^{(0)}-c\partial_x)\phi^{(-1)}_{2}
+(\lambda^{(-1)}+\frac{T}{2\pi i} v) \phi^{(0)}_{2}),
\]
which is solvable by the Fredholm alternative, provided that 
\[
(\lambda^{(0)}-c\partial_x)\phi^{(-1)}_{1}-\frac{T}{2\pi i}\phi^{(0)}_{2}=0.
\]
Therefore 
\begin{equation}\label{eqn3:pencil}
((\lambda^{(0)}-c\partial_x)^2-(1+cv))\phi^{(-1)}_{1}=0.
\end{equation}
Similarly as in Section~\ref{sec2:spec}, if the Hill's differential equation
\begin{equation}\label{eqn3:Hill}
c^2y''-(1+cv)y=0
\end{equation}
is elliptic (in a band), so that \eqref{eqn3:pencil} has a bounded solution if and only if $\text{Re}(\lambda^{(0)})=0$, then the spectrum of \eqref{eqn3:spec} tends to infinity along the imaginary axis. If \eqref{eqn3:Hill} is hyperbolic (in a gap), on the other hand, so that \eqref{eqn3:pencil} has a bounded solution if and only if $\text{Re}(\lambda^{(0)})=\pm\sigma$, where $\sigma>0$ is the Lyapunov exponent of \eqref{eqn3:Hill}, then the spectrum tends to infinity along $\pm\sigma+i\mathbb{R}$. 

Similarly as in Section~\ref{sec2:Lame}, after the change of variables $x\mapsto\frac{x}{a}$, we can rewrite \eqref{eqn3:Hill} as the Lam\'e equation 
\begin{subequations}\label{eqn3:Lame}
\begin{equation}\label{eqn3:ell}
y''-4 m \sn^2(x,\sqrt{m}))y=\ell y,
\end{equation}
where 
\begin{equation}\label{def3:ell}
\ell=\frac{1}{a^2}+4(1+m-4mM(m))
\end{equation}
\end{subequations}
by \eqref{def3:sn} and \eqref{def3:c}. Let $\{\ell_j^P\}_{j=1}^\infty$ denote the periodic eigenvalues of \eqref{eqn3:ell} and $\{\ell_j^P\}_{j=1}^\infty$ the anti-periodic eigenvalues such that
\[
-\infty<\ell_1^P<\ell_1^A\leq\ell_2^A<\ell_2^P\leq\ell_3^P<\ell_3^A\leq\ell_4^A<\cdots.
\]
If \eqref{eqn3:Lame} lies in a band, that is, \eqref{def3:ell} is in
\[
(\ell_1^P,\ell_1^A)\cup (\ell_2^A,\ell_2^P)\cup (\ell_3^P,\ell_3^A)\cup \cdots,
\]
then the spectrum of \eqref{eqn3:spec} tends to infinity along the imaginary axis. If \eqref{eqn3:Lame} is in a gap, that is, 
\[
(-\infty,\ell_1^P)\cup (\ell_1^A,\ell_2^A)\cup (\ell_2^P,\ell_3^P) \cup \cdots,
\]
on the other hand, then the spectrum tends to infinity along $\pm\sigma+i\mathbb{R}$ for some $\sigma>0$. We refer to $(-\infty,\ell_1^P)$ as the zeroth gap for convenience. 

But one difference with Section~\ref{sec2:Lame} is that \eqref{eqn3:ell} is not a finite-gap Lam\'e equation, whereby analytical formulae of the eigenvalues seem not viable to use. Nevertheless we can compute the band edges numerically. 
Of practical usefulness to this end is \cite{Kiper}
\[
\sn^2(x,\sqrt{m})=M(m)-\frac{2\pi^2}{mK^2(\sqrt{m}))}
\sum_{k=1}^\infty\frac{kq^k}{1-q^{2k}}\cos(\frac{\pi kx}{K(\sqrt{m})}),
\]
where
\[
q=e^{-\frac{\pi K'(\sqrt{m})}{K(\sqrt{m})}},\quad
K'(\sqrt{m})=K(\sqrt{m'})\quad\text{and}\quad m'=1-m.
\]
Suppose that
\[
y(x)=\sum_{k\in\mathbb{Z}}\widehat{y}_k e^\frac{\pi ikx}{K},\quad
\text{$\widehat{y}_k$ are the Fourier coefficients},
\]
and we consider \eqref{eqn3:Lame} in the basis $\big\{e^\frac{\pi ikx}{K}: k\in\mathbb{Z}\big\}$ of $L^2(-K,K)$ subject to the periodic boundary condition. We make the Fourier collocation projection of \eqref{eqn3:Lame} to the subspace spanned by $\{e^\frac{\pi ikx}{K}:k\in(-N_k,N_k)\}$ (here $N_k=50$) and solve numerically the resulting eigenvalue problem for a diagonal matrix plus a Toepliz matrix, approximating numerically the periodic eigenvalues. We take the basis $\big\{e^{\frac{\pi ikx}{K}+\frac{\pi ix}{2K}}: k\in\mathbb{Z}\big\}$ of $L^2(-K,K)$ subject to $y(K)=-y(-K)$ for the anti-periodic eigenvalues. 

\begin{figure}[htbp]
\begin{center}
\includegraphics[width=0.495\textwidth]{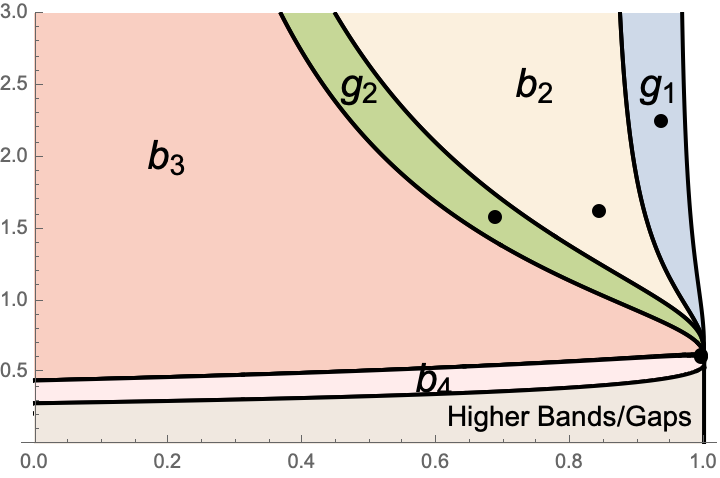}
\includegraphics[width=0.495\textwidth]{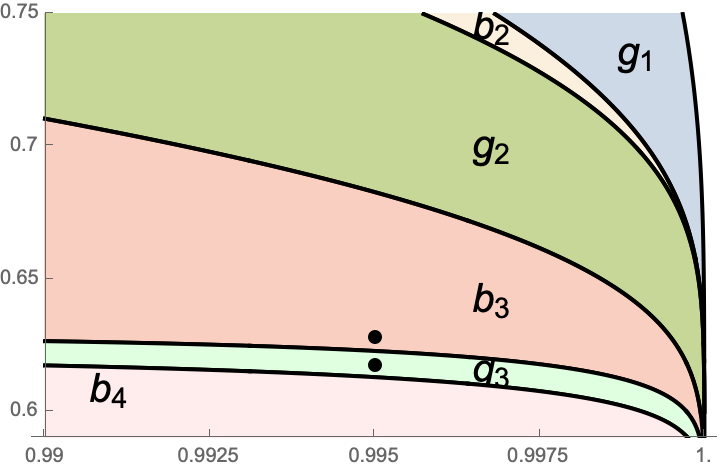}
\caption{$\square$ in the $(m,a)$ plane: $g_j$, $j\geq1,\in\mathbb{Z}$, denotes the region, for which \eqref{eqn3:Lame} is in the $j$-th gap, that is, \eqref{def3:ell} is in $(\ell_j^P, \ell_{j+1}^P)$ for $j$ even and $(\ell_j^A, \ell_{j+1}^A)$ for $j$ odd. Also $b_j$, $j\geq 2,\in\mathbb{Z}$, denotes the $j$-th band, that is, $(\ell_j^A,\ell_j^P)$ for $j$ even and $(\ell_j^P, \ell_j^A)$ for $j$ odd. The $j$-th gaps, $j\geq 4,\in\mathbb{Z}$, and the $j$-th bands, $j\geq 5,\in\mathbb{Z}$, are in the region labelled `higher bands and gaps'. 
The bullet points correspond to the values of $(m,a)$ for Figures~\ref{fig:Benney-Luke_Spine_g1}-\ref{fig:Benney-Luke_Spine_g3}. On the right is a close-up 
for the values of $(m,a)$ for Figures \ref{fig:Benney-Luke_Spine_b3} and \ref{fig:Benney-Luke_Spine_g3}.} 
\label{fig:BLBand_Gap}
\end{center}
\end{figure} 

Figure~\ref{fig:BLBand_Gap} shows bands and gaps of \eqref{eqn3:Lame} in the $(m,a)$ plane. Similarly to Section~\ref{sec2:Lame}, not all bands and gaps are present in $\square$. Specifically, the zeroth gap and the first band are not present. We expect that \eqref{eqn3:ell} has countably many gaps and all are open. There is no proof, however, to the best of the authors' knowledge. Figure~\ref{fig:BLBand_Gap} shows the $j$-th gaps, $j=1$, $2$, $3$, and the $j$-th bands, $j=2$, $3$, $4$. The $j$-th gaps, $j\geq4,\in\mathbb{Z}$, in the region labelled `higher bands and gaps', are exceedingly narrow in the $(m,a)$ plane, for which the spectrum of \eqref{eqn3:spec} would tend to infinity along $\pm\sigma+i\mathbb{R}$ for $\sigma$ exceedingly small. Such instability would not be of profound physical significance. 

We summarize our conclusion.

\begin{theorem}\label{thm:BL}
Let $(m,a)\in\square$ and $\mathbf{v}$ denotes a periodic traveling wave of \eqref{eqn3:vw}, $c$ the wave speed and $T$ the period, depending on $m$ and $a$. Let $\{\ell^P_j\}_{j=1}^\infty$ denote the periodic eigenvalues of \eqref{eqn3:ell} and $\{\ell^A_j\}_{j=1}^\infty$ the anti-periodic eigenvalues. If 
\begin{align*}
&\frac{1}{\sqrt{\ell^A_{2j}-4(1+m-4mM(m))}}<a<\frac{1}{\sqrt{\ell^A_{2j-1}-4(1+m-4mM(m))}}&& \text{for $j\geq2,\in\mathbb{Z}$},
\intertext{or}
&\frac{1}{\sqrt{\ell^P_{2j+1}-4(1+m-4mM(m))}}<a<\frac{1}{\sqrt{\ell^P_{2j}-4(1+m-4mM(m))}}&& \text{for $j\geq1,\in\mathbb{Z}$},
\end{align*}
so that \eqref{eqn3:Lame} is in a gap, then the spectrum of \eqref{eqn3:spec} tends to infinity along $\pm\sigma+i\mathbb{R}$ for some $\sigma>0$.
\end{theorem}

\subsection{Numerical experiments}\label{sec3:nuemrics}

Similarly as in Section~\ref{sec2:numerics}, we compute the spectrum of \eqref{eqn3:spec} numerically by a Fourier spectral method. One minor difference is that we integrate \eqref{eqn3:v0} numerically and approximate the Fourier coefficients by numerical quadrature, without recourse to an analytic formula for the Jacobi elliptic function. Similarly as in Section~\ref{sec2:numerics}, when \eqref{eqn3:Hill} is in a gap, we integrate \eqref{eqn3:Hill} numerically and approximate numerically the eigenvalues of the monodromy matrix. 

\begin{figure}[htbp]
\begin{center}
\includegraphics[width=0.495\textwidth]{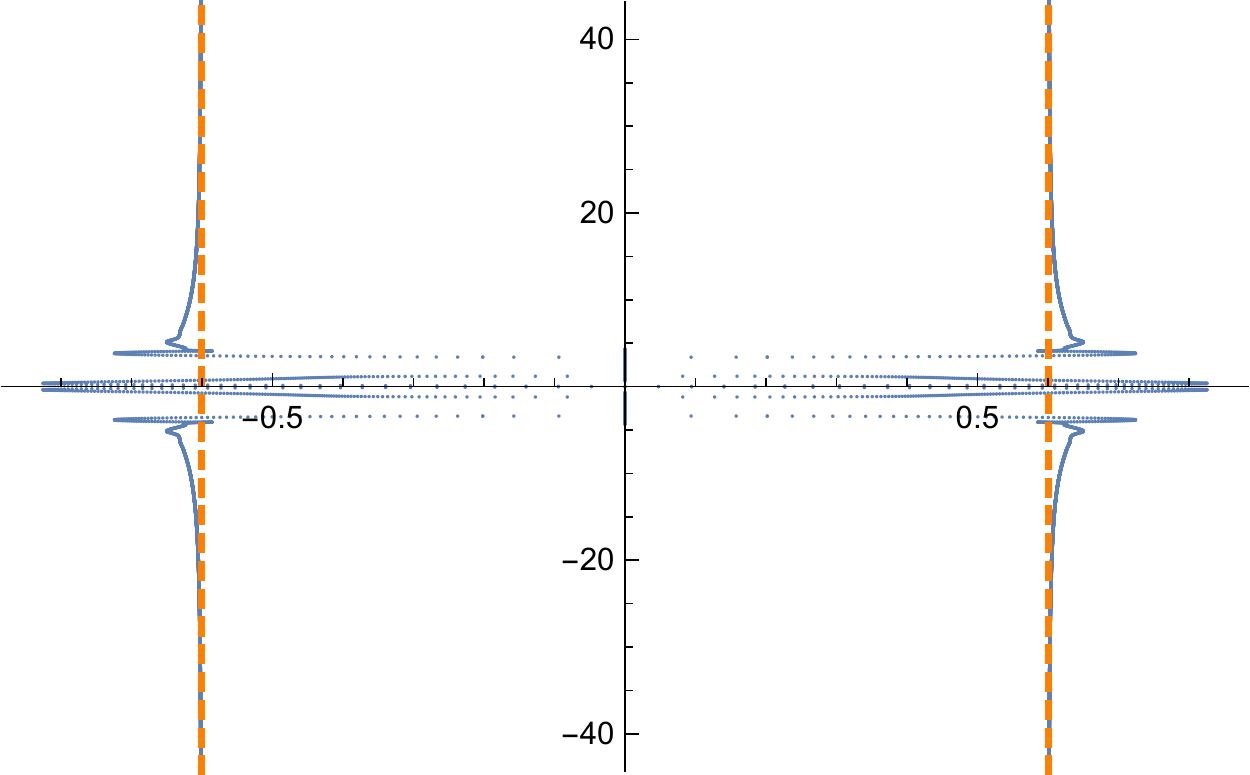}
\includegraphics[width=0.495\textwidth]{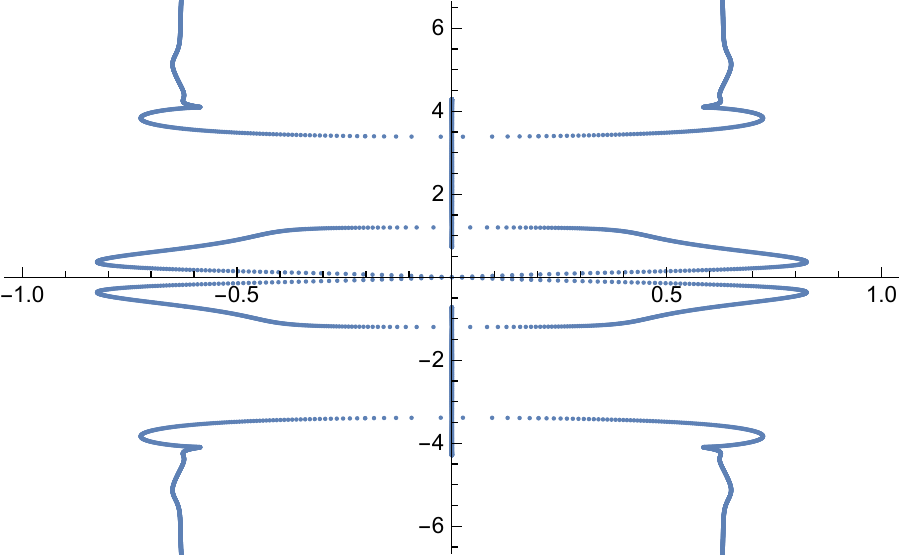}
\caption{The numerically computed spectrum of \eqref{eqn3:spec}, where $v$, $c$, $T$ depend on $m=0.9342$ and $a=2.25$, for which \eqref{eqn3:Lame} is in the first gap. The spectrum tends towards infinity along the dashed lines $\pm\sigma+i\mathbb{R}$, where $\sigma\approx0.600755$. On the right is a close-up for modulational instability near $0\in\mathbb{C}$.}
\label{fig:Benney-Luke_Spine_g1}
\end{center}
\end{figure}

\begin{figure}[htbp]
\begin{center}
\includegraphics[width=0.495\textwidth]{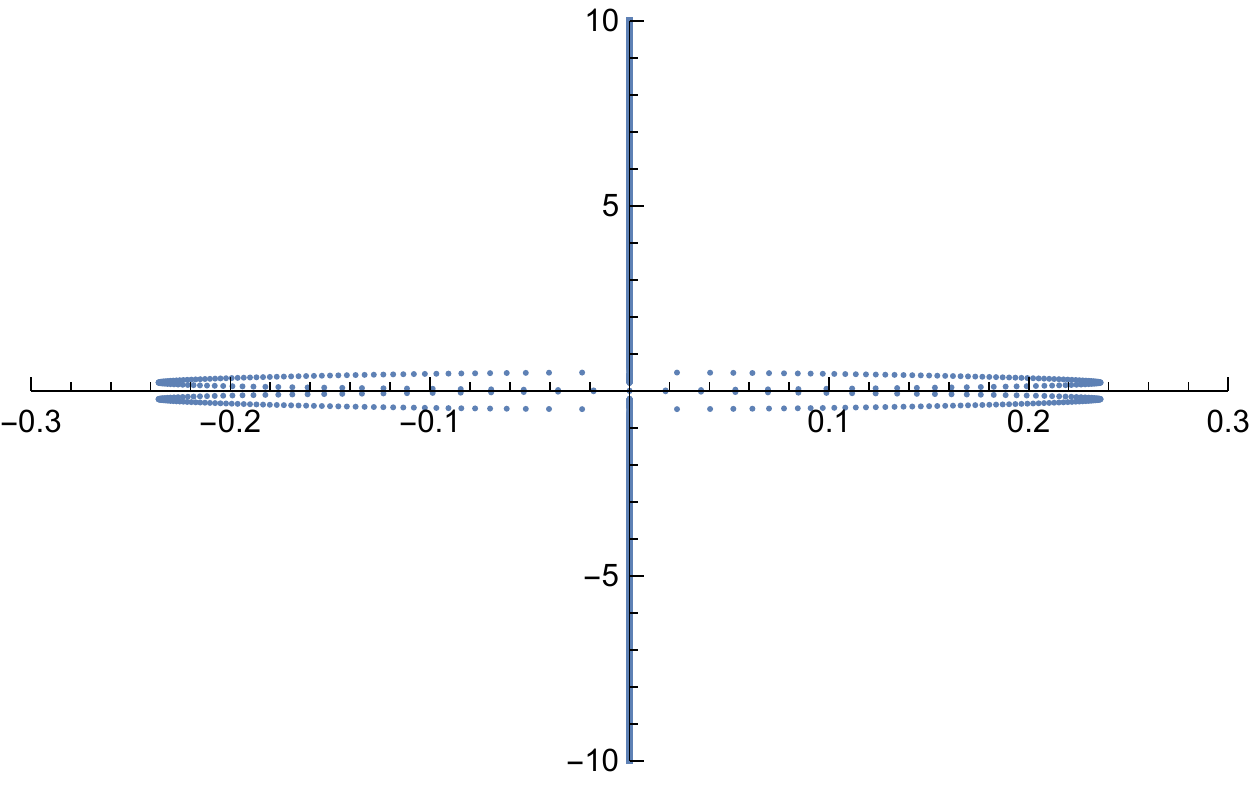}
\includegraphics[width=0.495\textwidth]{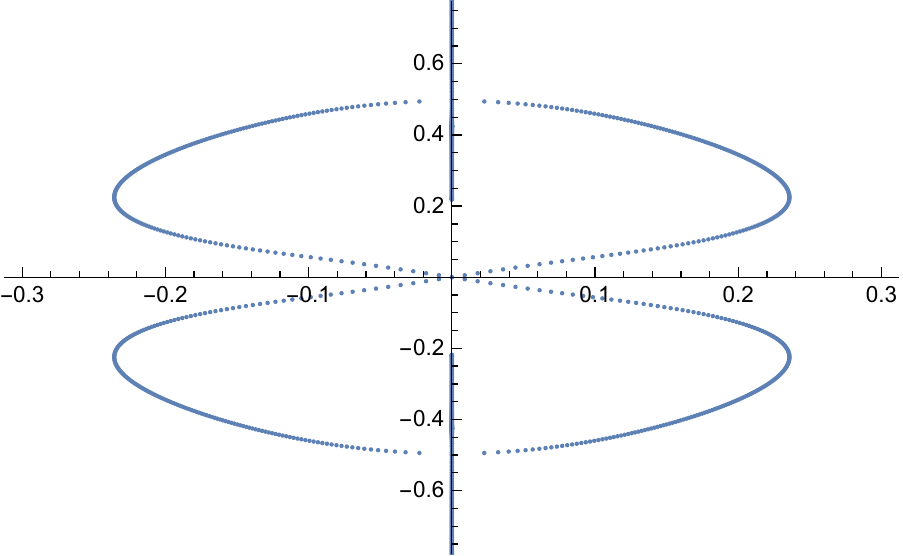}
\caption{The spectrum of \eqref{eqn3:spec} for $m=0.8428$ and $a=1.621$, for which \eqref{eqn3:Lame} is in the second band. The spectrum lies along the imaginary axis far away from $0\in\mathbb{C}$. On the right is a close-up for modulational instability. }
\label{fig:Benney-Luke_Spine_b2}
\end{center}
\end{figure}

\begin{figure}[htbp]
\begin{center}
\includegraphics[width=0.495\textwidth]{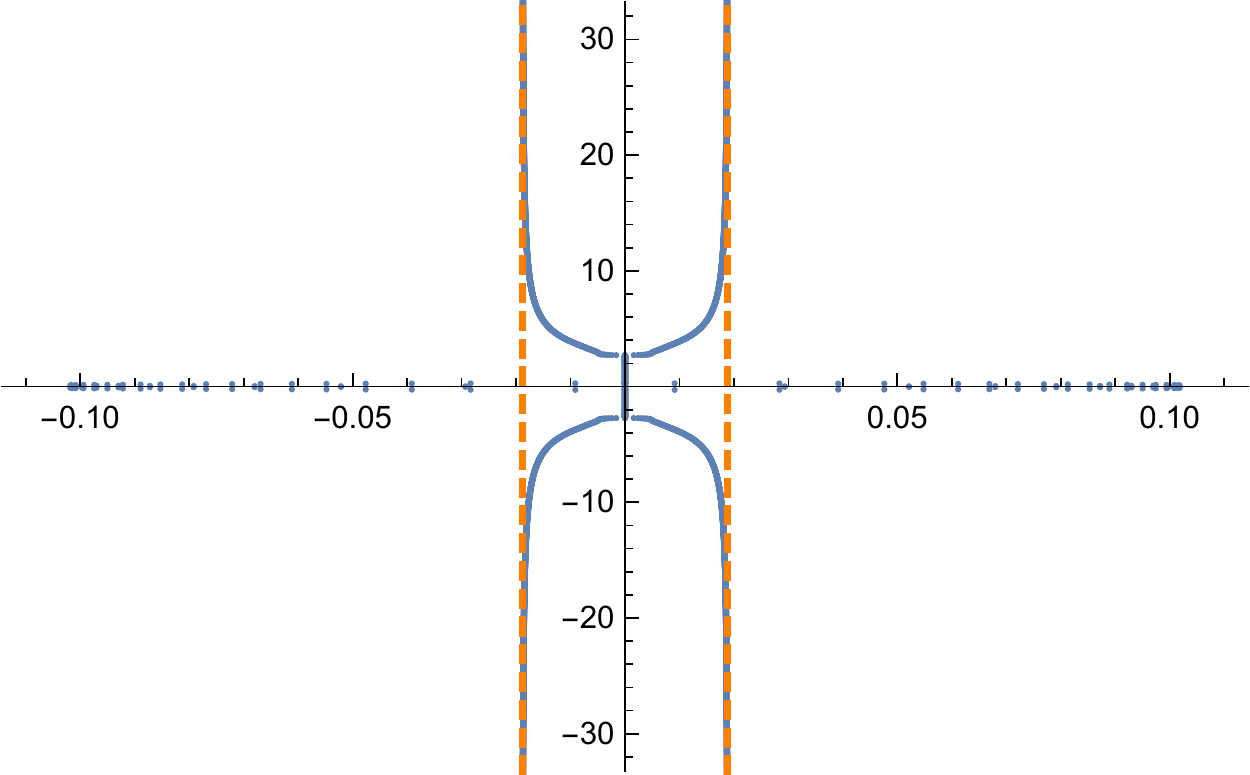}
\includegraphics[width=0.495\textwidth]{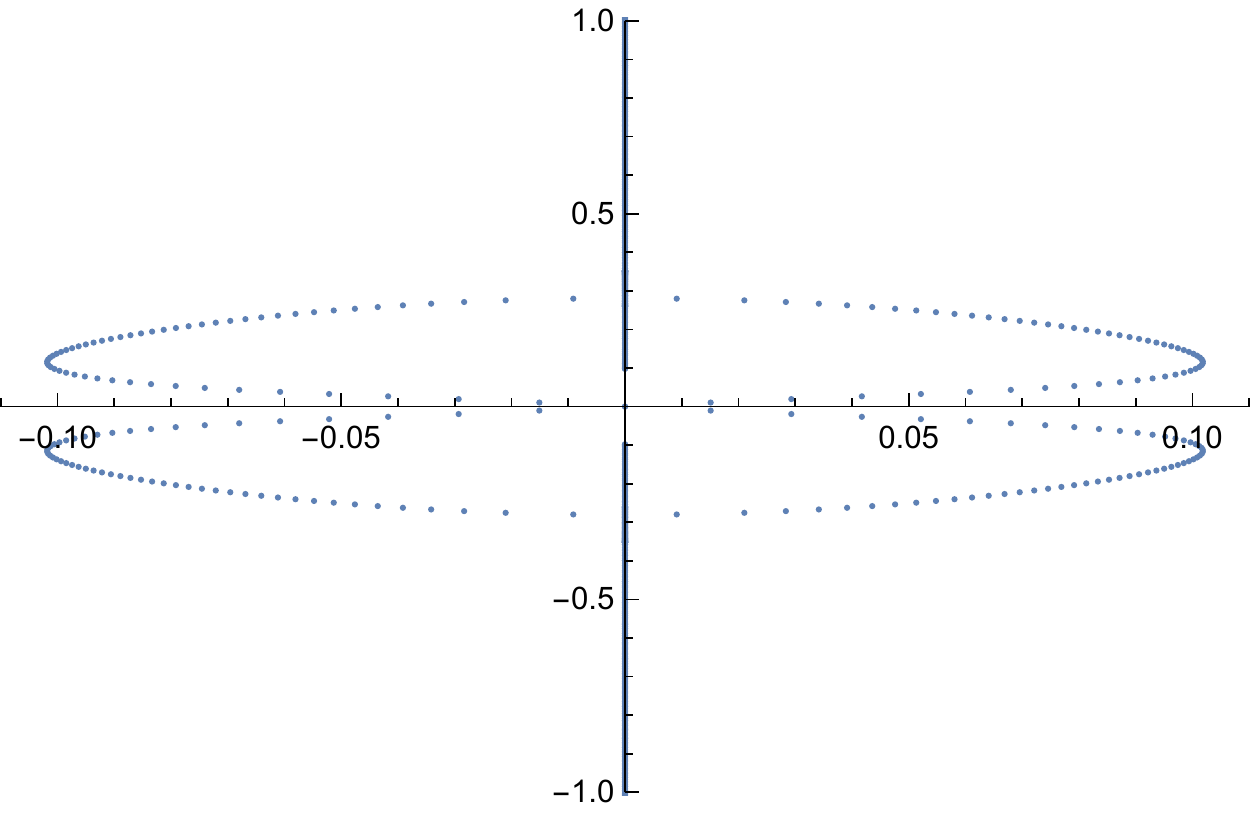}
\caption{The spectrum of \eqref{eqn3:spec} for $m=0.6872$ and $a=1.584$, for which \eqref{eqn3:Lame} is in the second gap. The spectrum tends towards infinity along the dashed lines $\pm\sigma+i\mathbb{R}$, where $\sigma\approx0.0188322$. On the right is a close up for modulational instability near $0\in\mathbb{C}$.} 
\label{fig:Benney-Luke_Spine_g2}
\end{center}
\end{figure}

\begin{figure}
\begin{center}
\includegraphics[width=0.495\textwidth]{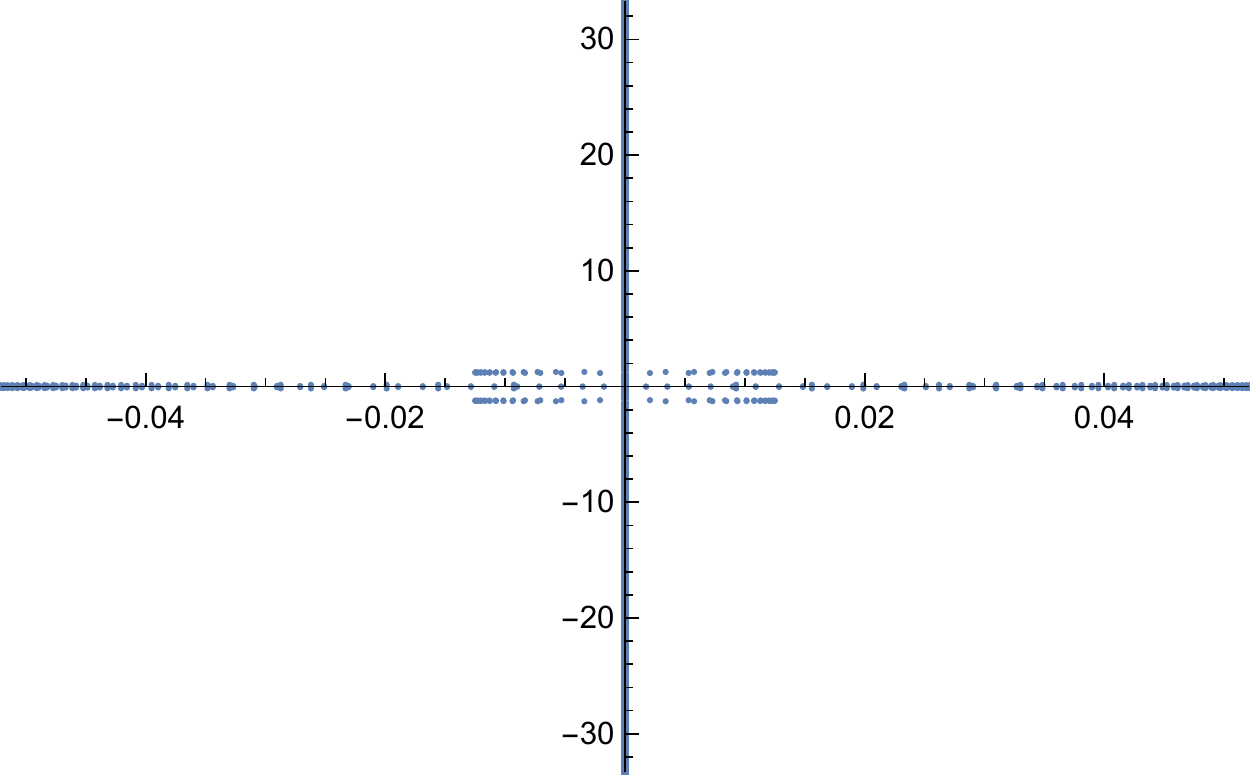}
\includegraphics[width=0.495\textwidth]{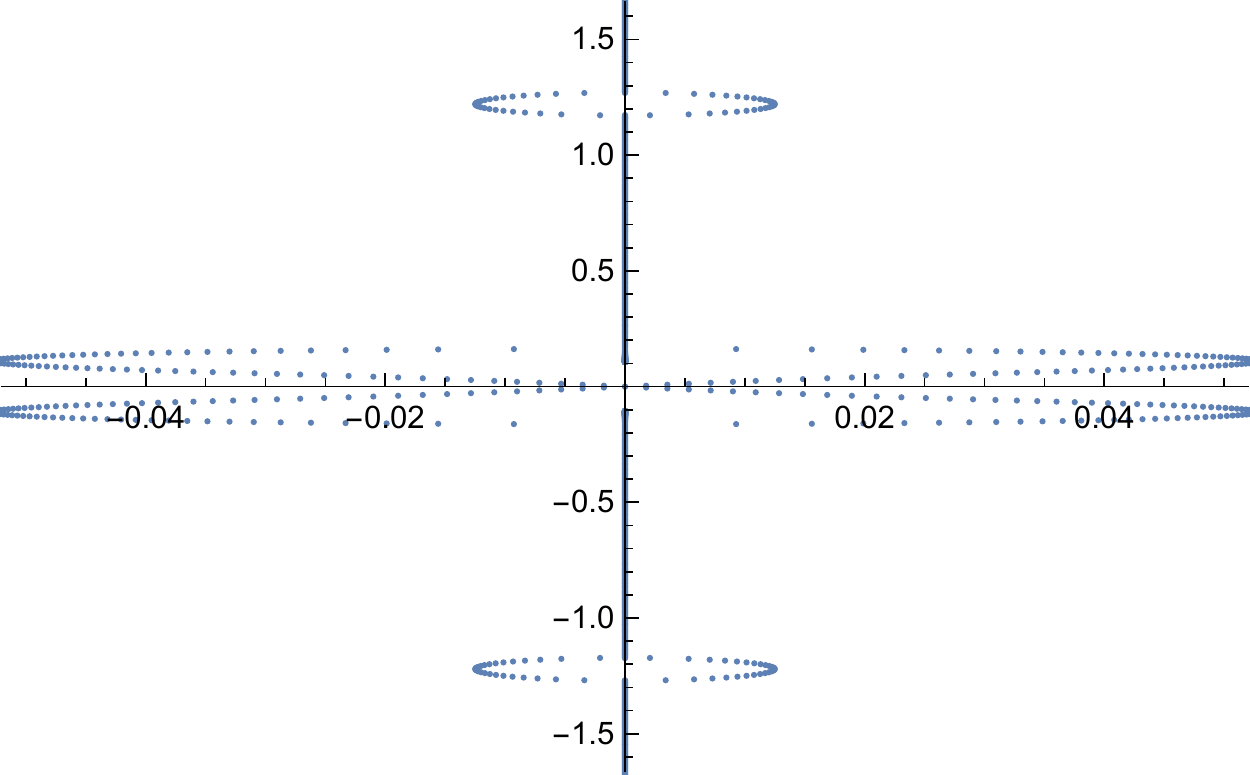}
\caption{The spectrum of \eqref{eqn3:spec} for $m=0.995$ and $a=0.628$, for which \eqref{eqn3:Lame} is in the third band. The spectrum lies along the imaginary axis far away from $0\in\mathbb{C}$. On the right is a close up for modulational instability near $0\in\mathbb{C}$ and finite wavelength instability near $\pm1.25i$. }
\label{fig:Benney-Luke_Spine_b3}
\end{center}
\end{figure}

\begin{figure}
\begin{center}
\includegraphics[width=0.495\textwidth]{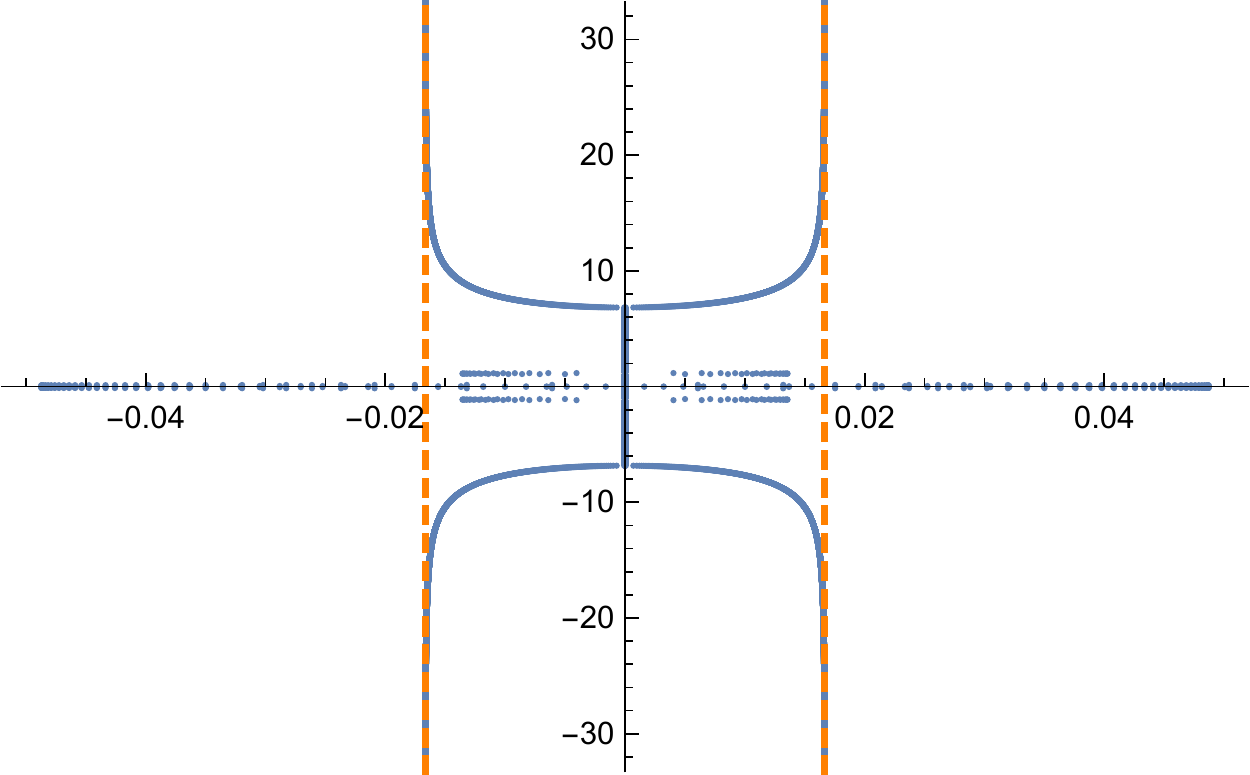}
\includegraphics[width=0.495\textwidth]{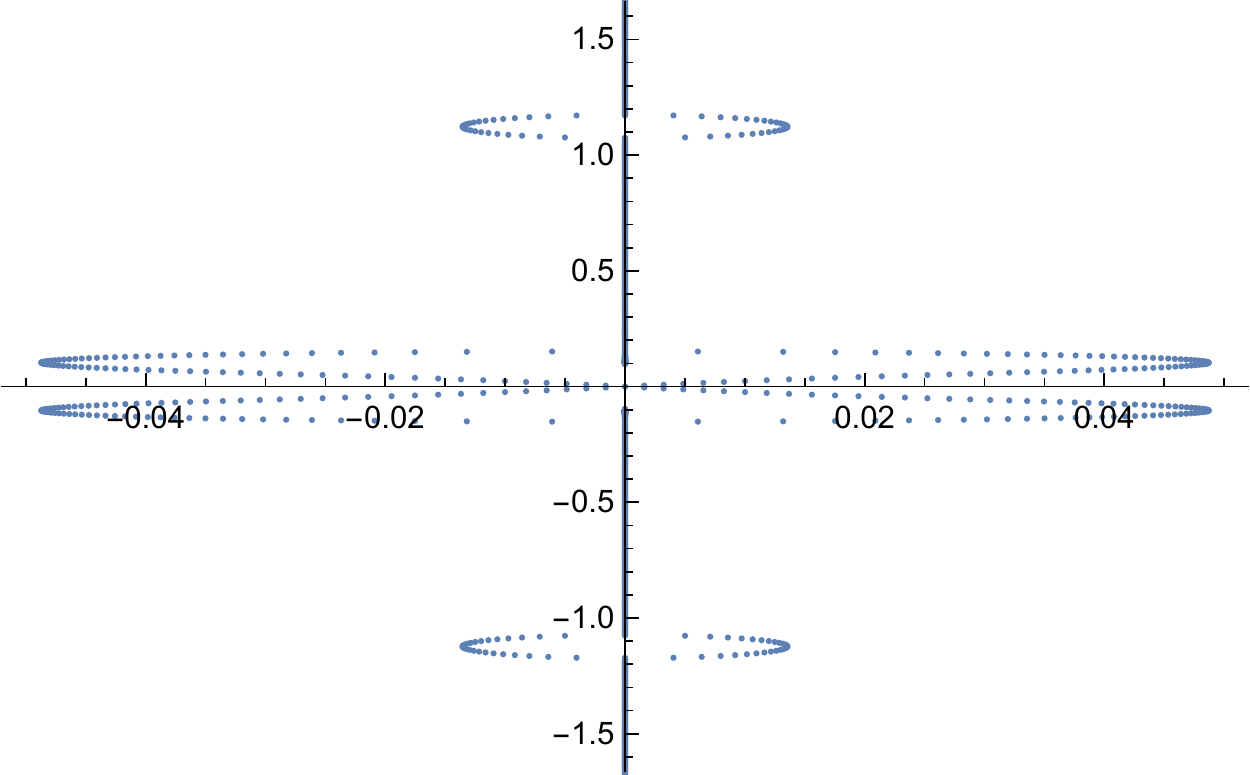}
\caption{The spectrum of \eqref{eqn3:spec} for $m=0.995$ and $a=0.618$, for which \eqref{eqn3:Lame} is in the third gap. The spectrum tends towards infinity along the dashed lines $\pm\sigma+i\mathbb{R}$, where $\sigma\approx0.016658$. On the right is a close-up for modulational instability near $0\in\mathbb{C}$ and finite wavelength instability near $\pm1.1i$.}
\label{fig:Benney-Luke_Spine_g3}
\end{center}
\end{figure}

Figures~\ref{fig:Benney-Luke_Spine_g1}, \ref{fig:Benney-Luke_Spine_g2}, \ref{fig:Benney-Luke_Spine_g3} provide examples of the numerically computed spectrum of \eqref{eqn3:spec}, for which \eqref{eqn3:Lame} lies in the $j$-th gap, $j=1$, $2$, $3$, respectively. The values of $m$ and $a$ correspond to the bullet points in the regions $g_j$ of Figure~\ref{fig:BLBand_Gap}, $j=1$, $2$, $3$. The numerical result corroborates Theorem~\ref{thm:BL} that the spectrum tends towards infinity along $\pm\sigma+i\mathbb{R}$ for some $\sigma>0$. In each of the examples, $\sigma$ agrees well with $\frac{\log(|\mu|)}{T}$, where $\mu$ is the numerically computed eigenvalue of the monodromy matrix of \eqref{eqn3:Hill} such that $|\mu|>1$. Notice modulational instability to long wavelength perturbations near $0\in\mathbb{C}$ in each example, and spectral instability to finite wavelength perturbations away from $0\in\mathbb{C}$ in Figure~\ref{fig:Benney-Luke_Spine_g3}.

Figures~\ref{fig:Benney-Luke_Spine_b2} and \ref{fig:Benney-Luke_Spine_b3} show examples of the spectrum, for which \eqref{eqn3:Lame} is in the $j$-th band, $j=2$ and $3$, respectively. Recall that there are no values of $m$ and $a$ for which \eqref{eqn3:Lame} is in the first band. The numerical result supports our analytical prediction that the spectrum tends towards infinity along the imaginary axis. Notice modulational instability in each of the examples and finite wavelength instability in Figure~\ref{fig:Benney-Luke_Spine_b3}. 

\section{The coupled Benjamin--Bona--Mahony system}\label{sec:abcd}

Last but not least, we turn to \eqref{eqn1:BCS}, where $a=c=0$ and $b=d=\frac16$, that is,
\begin{equation}\label{eqn4:BCS}
\begin{aligned}
&\eta_t+u_x-\frac{1}{6}\eta_{xxt}+(\eta u)_x=0, \\
&u_t+\eta_x-\frac{1}{6}u_{xxt}+uu_x=0,
\end{aligned}
\end{equation}
whose dispersion relation 
\[
\omega^2(k)=\frac{k^2}{(1+\frac16k^2)(1+\frac16k^2)}
\]
remains bounded for all $k\in\mathbb{R}$. We remark that $b=d=\frac16$ is for convenience and better correspondence to earlier works~\cite{BCS,BCS2}. 

\subsection{Parametrization of periodic traveling waves}\label{sec4:periodic}

A traveling wave of \eqref{eqn4:BCS} takes the form $\eta(x-ct-x_0)$ and $u(x-ct-x_0)$ for some $c\neq0,\in\mathbb{R}$ for some $x_0\in\mathbb{R}$, and they satisfy by quadrature
\begin{equation}\label{eqn4:BCS0}
\begin{aligned}
&-c\eta+u+\frac16c\eta''+u\eta=b_1, \\
&-cu+\eta+\frac16cu''+\frac12u^2=b_2
\end{aligned}
\end{equation}
for some $b_1$, $b_2\in\mathbb{R}$. Periodic solutions of \eqref{eqn4:BCS0} were treated in \cite{CCD,CCN}, among others, but for $b_1=b_2=0$, and not exhaustively, to the best of the authors' knowledge. Here we work out all $b_1$, $b_2\in\mathbb{R}$. Similarly as in Sections~\ref{sec2:periodic} and \ref{sec3:periodic}, we can mod out $x_0$. 

Eliminating $\eta$ from \eqref{eqn4:BCS0}, we arrive at
\begin{equation}\label{eqn4:u0}
c^2u''''+(u-c)(12cu''+18u^2-36cu+b_2)+6c(u')^2-36u+b_1=0.
\end{equation}
Suppose that a solution of \eqref{eqn4:u0} satisfies
\[
u''=P(u;c,b_1,b_2),\quad \text{$P$ is a polynomial of $u$},
\]
allowing for elliptic and hyperelliptic functions, whence
\begin{subequations}\label{eqn4:u'}
\begin{equation}\label{eqn4:Q}
\frac12(u')^2=Q(u;c,b_1,b_2),\quad Q'=P,
\end{equation}
and also
\[
u''''=2(P''Q+PP')(u;c,b_1,b_2). 
\]
Substituting into \eqref{eqn4:u0}, we deduce that $\deg(Q)\leq3$, so that no hyperelliptic functions. 
Assuming that $\deg(Q)=3$, after some algebra we find 
\begin{equation}\label{def4:Q}
Q(u;c,b_1,b_2)=-\frac{3}{5c}u^3+\frac95u^2-\frac{144c^2+25b_2-900}{330c}u-\frac{1008c^3-6300c+275b_1-100b_2c}{1320c}.
\end{equation}
\end{subequations}

One can work out the existence of non-constant periodic solutions of \eqref{eqn4:u'} and, hence, non-constant periodic traveling waves of \eqref{eqn4:BCS}, provided that the discriminant 
\begin{equation}\label{def4:disc}
\begin{aligned}
\disc(Q):=\frac{1}{76665c^4}(&4665600c^6+27993600c^4-363181968c^2\\
&-777600b_2c^4+43200b_2^2c^2-3110400b_2c^2+23287176b_1c \\
&-323433 b_1^2-800 b_2^3+86400 b_2^2-3110400 b_2+37324800)
\end{aligned}
\end{equation}
is positive. When $b_1=b_2=0$, \eqref{def4:disc} has simple roots at $c=\pm\frac{5}{2}$ and $\pm\frac12 \sqrt{\frac{33\sqrt{57}}{10}-\frac{49}{2}}\approx\pm 0.3219$, together with a pole at zero, so that \eqref{eqn4:u'} has periodic solutions, provided that
\[
c\in\Big(-\infty,-\frac52\Big)\bigcup\bigg(-\frac12\sqrt{\frac{33\sqrt{57}}{10}-\frac{49}{2}},0\bigg) 
\bigcup\bigg(0,\frac12\sqrt{\frac{33\sqrt{57}}{10}-\frac{49}{2}}\bigg)\bigcup\Big(\frac52,\infty\Big). 
\]
This reproduces the result of \cite{CCN}. 

More generally, \eqref{eqn4:u'} has periodic solutions, provided that $c$ is in some interval, depending on $b_1$ and $b_2$, for which \eqref{def4:disc} is positive. The number of intervals of such `admissible' wave speeds is constant in open sets in the $(b_1,b_2)$ plane and changes across the curves of co-dimension one, which can be found in closed form: 
\begin{equation}\label{def4:curve123}
\begin{aligned}
&b_1^2=\frac{800}{323433}(36-b_2)^3, \\
&b_1^2=\frac{16}{1617165}\Big(-125 b_2^3+13500 b_2^2+4365495b_2-168821820
+\sqrt{25 b_2^2-1800 b_2+35583}^3\Big), \\
&b_1^2=72b_2-2592.
\end{aligned}
\end{equation}

\begin{figure}[htbp]
\begin{center}
\includegraphics[width=0.5\textwidth]{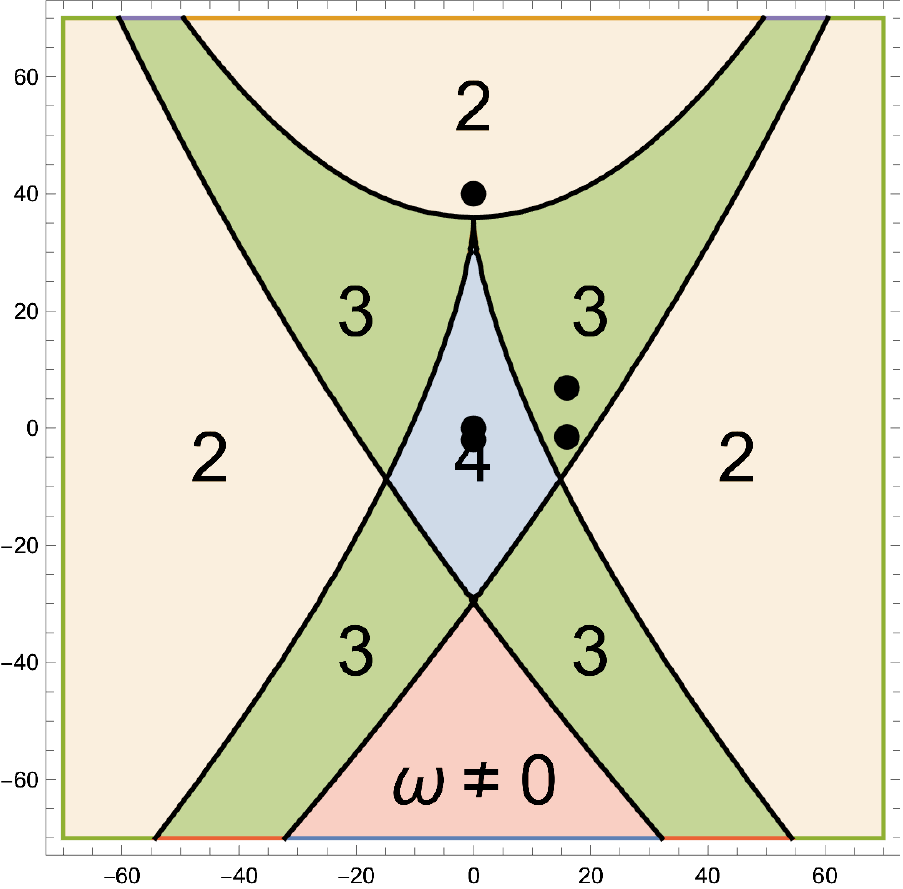}
\caption{Regions in the $(b_1,b_2)$ plane, separated by the curves in \eqref{def4:curve123}, in each of which the number of intervals of admissible wave speeds is constant and denoted by the label. In the region $c\neq0$, the intervals of admissible wave speeds are $(-\infty,0)$ and $(0,\infty)$. The bullet points correspond to the values of $b_1$ and $b_2$ for Figures \ref{fig4:R4}--\ref{fig4:R2}. }
\label{fig:PhaseDiagram}
\end{center}
\end{figure}

Figure~\ref{fig:PhaseDiagram} shows the regions in the $(b_1,b_2)$ plane, in each of which the number of intervals of admissible wave speeds is constant, denoted by the label, and whose boundaries are made up of the curves in \eqref{def4:curve123}. For any $b_1$, $b_2\in\mathbb{R}$, \eqref{def4:disc} is positive when $|c|$ becomes sufficiently large, so that \eqref{eqn4:u'} has a periodic solution, whereas no solutions when $c=0$. Therefore the number of intervals of admissible wave speeds is $\geq2$. 

For $(b_1,b_2)$ in the region~$4$ of Figure~\ref{fig:PhaseDiagram}, there are four intervals of admissible wave speeds, two where $c>0$ and two where $c<0$. They are symmetric about $c=0$ when $b_1=0$ and asymmetric otherwise. Particularly, $(b_1,b_2)=(0,0)$ is in the region~$4$, which was treated in \cite{CCN,CCD} and others. For $(b_1,b_2)$ in the region~$3$, the number of intervals of admissible wave speeds is three, and in the remaining regions, it is two. In the region $c\neq0$, where $b_1$ and $b_2$ satisfy 
\[
b_1^2<\frac{16}{1617165}\Big(-125b_2^3+13500b_2^2+4365495b_2-168821820+\sqrt{25b_2^2-1800b_2+355833}^3\Big)
\] 
and $b_2<\frac95(20-11\sqrt{11})$, the number of intervals of admissible wave speeds is two, but we single it out because the intervals are $(-\infty,0)$ and $(0,\infty)$. In other words, the only inadmissible wave speed is zero. In all the other regions, by contrast, there is a (nonempty) closed interval of inadmissible wave speeds. 

Actually, periodic solutions of \eqref{eqn4:u'} can be found in closed form in terms of the Jacobi elliptic functions. 
Recall that $y(x)=\frac{2m-1}{3}-m\cn^2(x,\sqrt{m})$ is a solution of 
\[
(y')^2=4y^3-\frac43(1-m+m^2)y-\frac{4}{27}(2-3m-3m^2+2m^3),
\]
where $m\in(0,1)$ is an elliptic operator. After appropriate changes of variables, notice that
\[
z(x)=z_0\qty(\frac{2m-1}{3}-m\cn^2(ax,\sqrt{m}))
\]
is a solution of 
\[
(z')^2=\alpha_3z^3+\alpha_1z+\alpha_0,
\]
where 
\begin{equation}\label{def4:a}
z_0=\frac{9\alpha_0(1-m+m^2)}{\alpha_1(2-3m-3m^2+2m^3)}\quad\text{and}\quad 
a=\sqrt{{\frac{9\alpha_0\alpha_3(1-m+m^2)}{4\alpha_1(2-3m-3m^2+2m^3)}}},
\end{equation}
if and only if 
\begin{equation}\label{eqn4:m}
-\frac{(2-3m-3m^2+2m^3)^2}{27(1-m+m^2)^3}=\frac{\alpha_0^2\alpha_3}{\alpha_1^3}\in\Big(-\frac{4}{27},0\Big)
\end{equation}
for some $m\in(0,1)$. The left hand side of \eqref{eqn4:m} increases monotonically from $-\frac{4}{27}$ at $m=0$ to $0$ at $m=\frac12$, and then decreases monotonically to $-\frac{4}{27}$ at $m=1$. Therefore whenever $\frac{\alpha_0^2\alpha_3}{\alpha_1^3}\in\qty(-\frac{4}{27},0)$, \eqref{eqn4:m} has exactly two solutions in the interval $(0,1)$, one less than $\frac12$ and one greater than $\frac12$. But $a$ (see \eqref{def4:a}) is real only for one of the solutions. If $\alpha_0>0$ then we must choose the solution of \eqref{eqn4:m} in the interval $(0,\frac12)$, and if $\alpha_0<0$ then $(\frac12,1)$. 

Returning to \eqref{eqn4:u'}, let $u=v+c$, and after some algebra we arrive at
\[
(v')^2=-\frac{6}{5c}v^3+\frac{90c^2-5b_2+180}{66c}v+\frac{-5b_1+180}{24c}
=:\alpha_3v^3+\alpha_1v+\alpha_0,
\]
so that whenever \eqref{eqn4:m} holds true for some $m\in(0,1)$, depending on $c$, $b_1$, $b_2$,
\[
u(x)=c+u_0\qty(\frac{2m-1}{3}-m\cn^2(ax,\sqrt{m}))
\]
gives a periodic solution of \eqref{eqn4:u'}, where ${\displaystyle u_0=\frac{9\alpha_0(1-m+m^2)}{\alpha_1(2-3m-3m^2+2m^3)} }$ and $a$ is in \eqref{def4:a}, whose period is $T=\frac{2K(\sqrt{m})}{a}$, where $K(\sqrt{m})$ is the complete elliptic integral of the first kind.

\begin{remark*}
Periodic solutions of \eqref{eqn4:u'} can also be found in terms of the Weierstrass elliptic functions as
\[
u(x)=c-\frac{10}{3}c\wp\qty(x;\frac{18c^2-b_2+36}{22c^2},\frac{-108c+3b_1}{80c^3}),
\]
depending on $c$, $b_1$, $b_2$, where $\wp(x;g_2,g_3)$ denotes the Weierstrass $\wp$-function, and $g_2$ and $g_3$ are the elliptic invariants. When $b_1=b_2=0$, this reproduces the result of \cite{CCD}. But the account of \cite{CCD}, while mathematically correct, could be  misleading. The solution of \cite{CCD} is given as
\[
u(x;c,\Lambda,\delta)=c-\frac{10}{3}\Lambda^2\delta^2c\wp\qty(\delta \Lambda x;\frac{9(c^2+2)}{11\Lambda^4\delta^4c^2},-\frac{27}{20\Lambda^6 \delta^6 c^2}) 
\]
(after correcting a very minor typo in \cite[(B.7)]{CCD}), where $\delta$ is described as free. But the Weierstrass $\wp$-function enjoys the scaling invariance
\[
\lambda^2\wp\qty(\lambda x;\frac{g_2}{\lambda^4},\frac{g_3}{\lambda^6})=\wp(x;g_2,g_3)
\]
for any $\lambda\neq0,\in\mathbb{R}$ (see \cite[pp. 439]{WW} or \cite[(1.41)]{pastras}, for instance). Therefore, the solutions of \cite{CCD} are independent of $\Lambda$ and $\delta$, and depend only on $c$ (plus $x_0$, but one can mod out $x_0$). 
\end{remark*}

\subsection{Asymptotic spectral analysis to short wavelength perturbations}\label{sec4:spec}

Let $\eta$ and $u$ denote a periodic traveling wave of \eqref{eqn4:BCS}, $c$ the wave speed, and $T$ the period, whose existence has been established in the previous subsection. Linearizing \eqref{eqn4:BCS} about $\eta$ and $u$ in the frame of reference moving at the speed $c$, and seeking a solution of the form $e^{\lambda t}\boldsymbol{\phi}(x)$, say, where $\lambda\in\mathbb{C}$, we arrive at 
\begin{equation}\label{eqn4:spec}
\lambda\boldsymbol{\phi}=
\mqty(c\partial_x+\partial_x \qty(\frac16\partial_{x}^2-1)^{-1} u &
\partial_x\qty(\frac16 \partial_{x}^2-1)^{-1}(1+\eta) \\
\partial_x\qty(\frac16 \partial_{x}^2-1)^{-1} & 
c\partial_x+\partial_x\qty(\frac16 \partial_{x}^2-1)^{-1} u)\boldsymbol{\phi}.
\end{equation}
Similarly as in Sections~\ref{sec2:spec} and \ref{sec3:spec}, $\eta$ and $u$ are spectrally stable if and only if the spectrum of \eqref{eqn4:spec} is contained in the imaginary axis, and $\lambda$ is in the spectrum if and only if \eqref{eqn4:spec} has a nontrivial bounded solution such that $\boldsymbol{\phi}(x+T)=e^{\frac{2\pi ikx}{T}}\boldsymbol{\phi}(x)$ for some $k\in\mathbb{R}$.

Similarly as in Sections~\ref{sec2:spec} and \ref{sec3:spec}, let
\[
\lambda=\lambda^{(1)}k+\lambda^{(0)}+\lambda^{(-1)}k^{-1}+\cdots\quad\text{and}\quad
\boldsymbol{\phi}(x)=e^{\frac{2\pi ikx}{T}}(\boldsymbol{\phi}^{(0)}(x)+\boldsymbol{\phi}^{(-1)}(x)k^{-1}+\boldsymbol{\phi}^{(-2)}(x)k^{-2}+\cdots)
\]
as $|k|\to\infty$ for some $\lambda^{(1)}$, $\lambda^{(0)}$, $\lambda^{(-1)}, \ldots\in\mathbb{C}$ for some $\boldsymbol{\phi}^{(0)}$, $\boldsymbol{\phi}^{(-1)}$, $\boldsymbol{\phi}^{(-2)}, \ldots\in L^\infty(\mathbb{R})\times L^\infty(\mathbb{R})$, so that  \eqref{eqn4:spec} becomes
\begin{multline*}
(\lambda^{(1)}k+\lambda^{(0)}+\lambda^{(-1)}k^{-1}+\cdots)
(\boldsymbol{\phi}^{(0)}+\boldsymbol{\phi}^{(-1)}k^{-1}+\boldsymbol{\phi}^{(-2)}k^{-2}+\cdots)\\
=:(\mathbf{L}^{(1)}k+\mathbf{L}^{(0)}+\mathbf{L}^{(-1)}k^{-1}+\cdots) (\boldsymbol{\phi}^{(0)}+\boldsymbol{\phi}^{(-1)}k^{-1}+\boldsymbol{\phi}^{(-2)}k^{-2}+\cdots)
\end{multline*}
as $|k|\to\infty$, where
\[
\mathbf{L}^{(1)}=\mqty(\frac{2\pi ic}{T} & 0 \\ 0 & \frac{2\pi ic}{T}), \qquad 
\mathbf{L}^{(0)}=\mqty(c\partial_x & 0 \\ 0 & c\partial_x)\quad\text{and}\quad 
\mathbf{L}^{(-1)}=\mqty(-\frac{3Ti}{\pi} & -\frac{3Ti}{\pi}(1+\eta) \\
-\frac{3Ti}{\pi}  & -\frac{3Ti}{\pi}u).
\]

At the order of $k$, we gather
\[
\lambda^{(1)}\boldsymbol{\phi}^{(0)}=\mathbf{L}^{(1)}\boldsymbol{\phi}^{(0)}
=\mqty(\frac{2\pi ic}{T} & 0 \\ 0 & \frac{2\pi ic}{T})\boldsymbol{\phi}^{(0)},
\]
whence 
\[
\lambda^{(1)}=\frac{2\pi ic}{T}\quad\text{and}\quad
\text{$\boldsymbol{\phi}^{(0)}$ is bounded and otherwise arbitrary}.
\]
At the order of $1$,
\[
\lambda^{(1)}\boldsymbol{\phi}^{(-1)}+\lambda^{(0)}\boldsymbol{\phi}^{(0)}
=\mathbf{L}^{(1)}\boldsymbol{\phi}^{(-1)}+\mathbf{L}^{(0)}\boldsymbol{\phi}^{(0)}.
\]
Since $\mathbf{L}^{(1)}-\lambda^{(1)}\mathbf{1}=\mathbf{0}$, where $\mathbf{1}$ denotes the identity operator, 
\[
\lambda^{(0)}\boldsymbol{\phi}^{(0)}=\mathbf{L}^{(0)}\boldsymbol{\phi}^{(0)}
=\mqty(c\partial_x & 0 \\ 0 & c\partial_x)\boldsymbol{\phi}^{(0)},
\]
whence $\boldsymbol{\phi}^{(0)}(x)=e^{c^{-1}\lambda^{(0)}x}\boldsymbol{\phi}$ for some constant $\boldsymbol{\phi}$. Seeking a bounded solution, we deduce that $\lambda^{(0)}$ is purely imaginary. On the other hand, $\lambda\sim\lambda^{(1)}k=\frac{2\pi ick}{T}$ as $|k|\to\infty$ to leading order, and $\lambda^{(0)}$ amounts to translate such---otherwise arbitrary---spectrum along the imaginary axis. Therefore, without loss of generality, 
\[
\lambda^{(0)}=0\quad\text{and}\quad
\text{$\boldsymbol{\phi}^{(0)}$ is a constant}.
\]
 
At the order of $k^{-1}$, similarly,	
\[
\lambda^{(1)}\boldsymbol{\phi}^{(-2)}+\lambda^{(0)}\boldsymbol{\phi}^{(-1)}+ \lambda^{(-1)}\boldsymbol{\phi}^{(0)}
=\mathbf{L}^{(1)}\boldsymbol{\phi}^{(-2)}+\mathbf{L}^{(0)}\boldsymbol{\phi}^{(-1)}+ \mathbf{L}^{(-1)}\boldsymbol{\phi}^{(0)}.
\]
Since $\mathbf{L}^{(1)}-\lambda^{(1)}\mathbf{1}=\mathbf{0}$ and since $\lambda^{(0)}=0$, 
\[
\mqty(c\partial_x & 0 \\ 0 & c\partial_x)\boldsymbol{\phi}^{(-1)}
=\mqty(\lambda^{(-1)}+\frac{3Ti}{\pi}u & \frac{3Ti}{\pi}(1+\eta) \\
\frac{3Ti}{\pi}  & \lambda^{(-1)}+\frac{3Ti}{\pi}u)\boldsymbol{\phi}^{(0)},
\]
which is solvable by the Fredholm alternative, provided that 
\[
\det \mqty(\lambda^{(-1)}+\frac{3Ti}{\pi}\overline{u} & \frac{3Ti}{\pi}(1+\overline{\eta}) \\
\frac{3Ti}{\pi} & \lambda^{(-1)}+\frac{3Ti}{\pi}\overline{u})=0,
\]
where
\begin{equation}\label{def4:means}
\overline{\eta}=\frac1T\int^T_0\eta(x)~\dd{x}\quad\text{and}\quad
\overline{u}=\frac1T\int^T_0u(x)~\dd{x}
\end{equation}
are the means of $\eta$ and $u$ over the period. 
Therefore
\[
\lambda^{(-1)}=-\frac{3Ti}{\pi}\overline{u}\pm\frac{3T}{\pi}\sqrt{-(1+\overline\eta)},
\]
and if $1+\overline{\eta}<0$, so that $\text{Re}(\lambda^{(-1)})\neq0$, then the spectrum of \eqref{eqn4:spec} tends to infinity along some curve whose real part is nonzero.

Surprisingly, 
\begin{equation}\label{eqn4:1+eta}
1+\overline{\eta}=\frac{4}{33}(2+c^2)+\frac{1213}{1188}b_2,
\end{equation}
not involving elliptic integrals, so that $1+\overline{\eta}<0$ when $b_2<-\frac{288}{1213}\approx-0.2374$. 
For a proof of \eqref{eqn4:1+eta}, see below. 

We summarize our conclusion.

\begin{theorem}\label{thm:BCS}
Let $\eta$ and $u$ denote a periodic traveling wave of \eqref{eqn4:BCS}, $c$ the wave speed, and $T$ the period, where $c$ is in an interval of admissible wave speeds, depending on $b_1$ and $b_2$. The spectrum of \eqref{eqn4:spec} satisfies
\begin{equation}\label{eqn4:lambda}
\lambda=\frac{2\pi ic}{T}k+\qty(-\frac{3Ti}{\pi}\overline{u}\pm\frac{3T}{\pi}
\sqrt{-\frac{4}{33}(2+c^2)-\frac{1213}{1188}b_2})k^{-1}+O(k^{-2})\quad\text{as $|k|\to\infty$},
\end{equation}
where $\overline{u}$ is in \eqref{def4:means}. Particularly, if
\begin{equation}\label{eqn4:cb2}
b_2<-\frac{144}{1213}(2+c^2)
\end{equation}
then the spectrum tends to infinity along some curve whose real part is nonzero.
\end{theorem}

We remark that $\overline{u}$ can be expressed in terms of elliptic integrals. But we will not give details here because it does not influence the quantitative result.

\begin{proof}[Proof of \eqref{eqn4:1+eta}]
Integrating the second equation of \eqref{eqn4:BCS0} over the period, we arrive at 
\[
\int_0^T (1+\eta(x))~\dd{x}=\int_0^T \qty(1+b_2-\frac12u^2+cu)~\dd{x}. 
\]
Recalling \eqref{eqn4:Q}, we rewrite 
\begin{equation}\label{eqn:mean1}
\int_0^T (1+\eta(x))~\dd{x}=\oint_\varGamma \qty(1+b_2-\frac12u^2+cu)~\frac{\dd{u}}{\sqrt{2 Q(u)}},
\end{equation}
where $\varGamma$ is a Jordan curve in the complex plane containing the range of $1+\eta$, an interval of $\mathbb{R}$. 
Recalling \eqref{def4:Q}, moreover, 
\begin{equation}\label{eqn:elliptident}
0=\oint_\varGamma \frac{Q'(u)~\dd{u}}{\sqrt{2 Q(u)}}=
\oint_\varGamma \qty(-\frac{9}{5c}u^2+\frac{18}{5}u-\frac{144c^2+25b_2-900}{330c})~\frac{\dd{u}}{\sqrt{2 Q(u)}}.
\end{equation}
Substituting \eqref{eqn:elliptident} into \eqref{eqn:mean1}, after some algebra we find
\[
\int_0^T (1+\eta(x))~\dd{x}
=\oint_\varGamma \qty(1+b_2+\frac{5}{18}\qty(\frac{144c^2+25 b_2-900}{330}))~\frac{\dd{u}}{\sqrt{2 Q(u)}}.
\]
On the other hand, 
\[
T=\oint_\varGamma\frac{\dd{u}}{\sqrt{2Q(u)}}.
\]
Therefore
\[
\frac{1}{T}\int_0^T (1+\eta(x))~\dd{x}=1+b_2+\frac{5}{18}\qty(\frac{144c^2+25 b_2-900}{330})
=\frac{4}{33}(2+c^2)+\frac{1213}{1188}b_2.
\]
This completes the proof.
\end{proof}

\subsection{Numerical experiments}\label{sec4:numerics}

Similarly as in Sections~\ref{sec2:numerics} and \ref{sec3:nuemrics}, we compute the spectrum of \eqref{eqn4:spec} numerically. When \eqref{eqn4:cb2} holds true, we numerically evaluate \eqref{eqn4:lambda} up to the order of $k^{-1}$. 

\begin{figure}[htbp]
\centering
\includegraphics[width=0.495\textwidth]{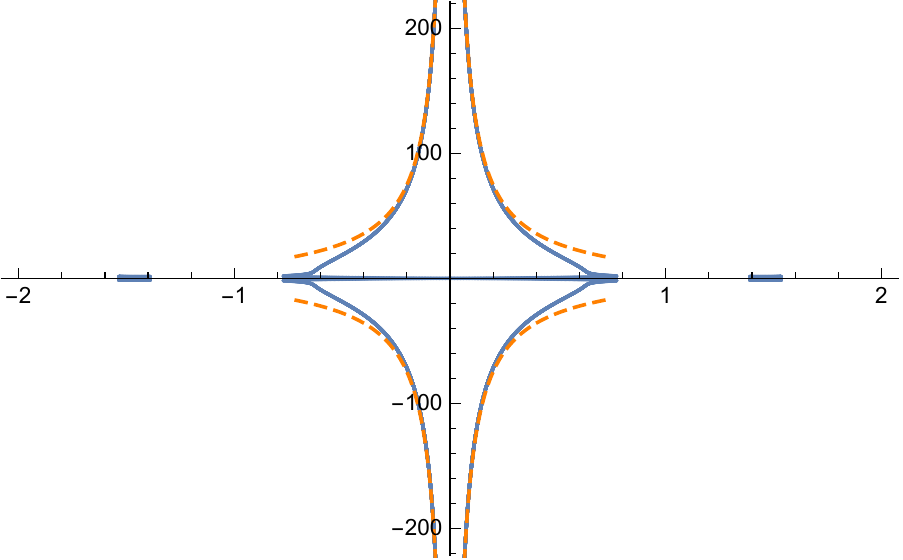}
\includegraphics[width=0.495\textwidth]{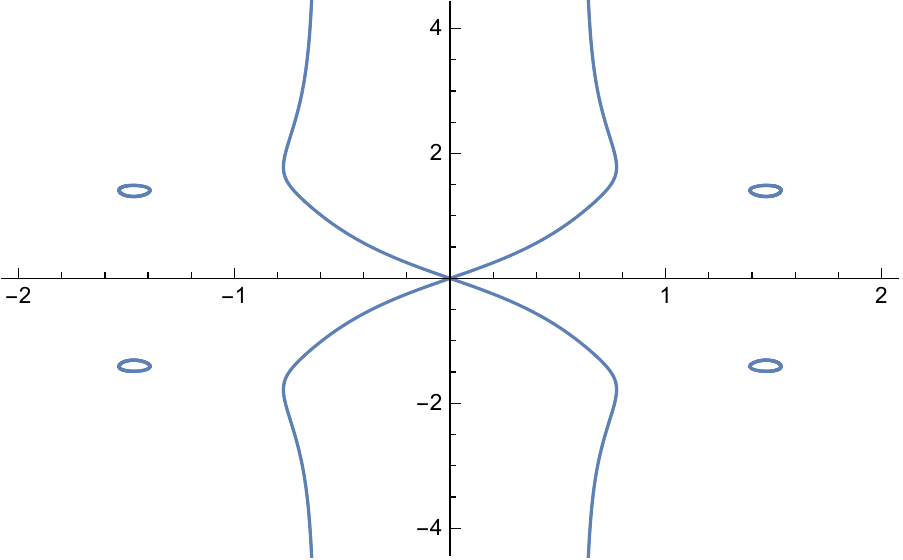}
\caption{The numerically computed spectrum of \eqref{eqn4:spec} for $c=3$, $b_1=0$, $b_2=-2$, which satisfy \eqref{eqn4:cb2}. The spectrum tends towards infinity along the dashed curves $\pm 3.56k^{-1}\pm(4.25k-18.43k^{-1})i$, $k\in\mathbb{R}$. On the right is a close-up for modulational instability near $0\in\mathbb{C}$ and, additionally, four isolated and closed loops of the spectrum centered at $\approx\pm 1.47\pm 1.39i$.}
\label{fig4:R4}
\end{figure}

Figure~\ref{fig4:R4} provides an example of the numerically computed spectrum of \eqref{eqn4:spec}. Here $b_1$ and $b_2$ correspond to a bullet point in the region $4$ of Figure~\ref{fig:PhaseDiagram}, and $c$ and $b_2$ satisfy \eqref{eqn4:cb2}. The spectrum tends towards infinity along the dashed curve, for which we evaluate \eqref{eqn4:lambda} numerically up to the order of $k^{-1}$. Therefore the numerical result corroborates Theorem~\ref{thm:BCS}. The right panel shows modulational instability near $0\in\mathbb{C}$ and, additionally, four isolated and closed loops of the spectrum whose real part is nonzero. Such closed loops of the spectrum are unusual, although isolated eigenvalues in the long wavelength limit, namely a solitary wave, can open up to a closed loop of spectrum for sufficiently large but finite wavelengths \cite{Gardner}. 

\begin{figure}[htbp]
\centering
\includegraphics[width=0.495\textwidth]{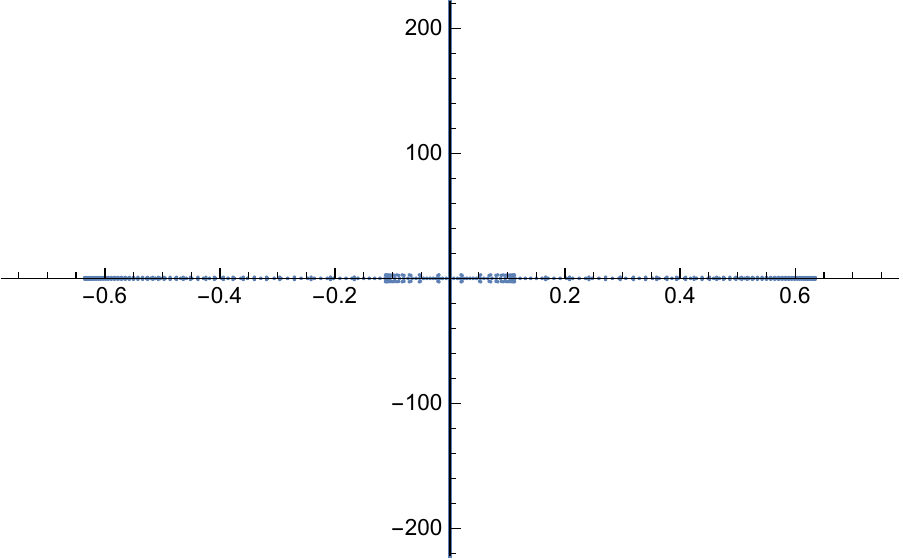}
\includegraphics[width=0.495\textwidth]{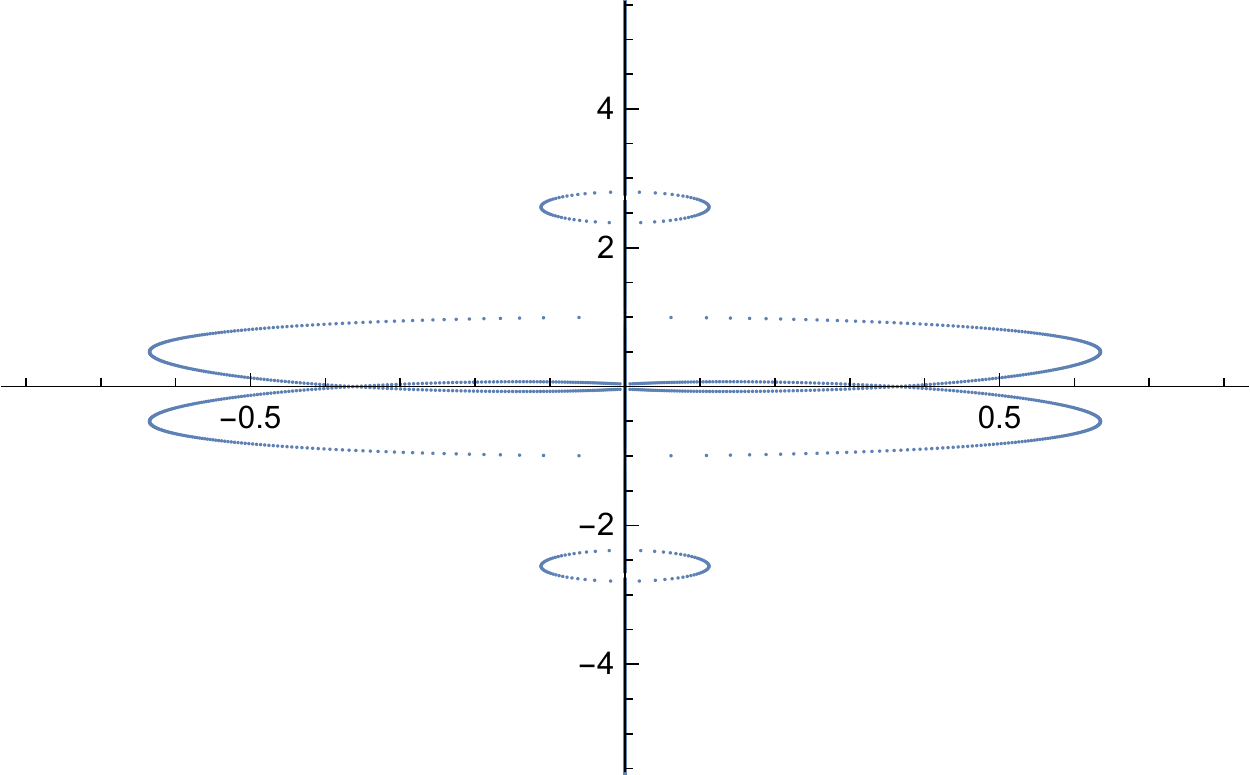}
\caption{The spectrum of \eqref{eqn4:spec} for $c=4$ and $b_1=b_2=0$, which do not satisfy \eqref{eqn4:cb2}. The spectrum lies along the imaginary axis towards infinity. On the right is a close up for modulational instability near $0\in\mathbb{C}$ and finite wavelength instability near $\pm2.6i$.}
\label{fig4:R40}
\end{figure}

Figure~\ref{fig4:R40} shows the spectrum of \eqref{eqn4:spec} for $b_1=b_2=0$ in the region $4$ of Figure~\ref{fig:PhaseDiagram}, and $c=4$ in an interval of admissible wave speeds, which do not satisfy \eqref{eqn4:cb2}. The right panel shows modulational and finite wavelength instability, but the spectrum lies along the imaginary axis otherwise, corroborating our analytical prediction. 

\begin{figure}[htbp]
\centering
\includegraphics[width=0.495\textwidth]{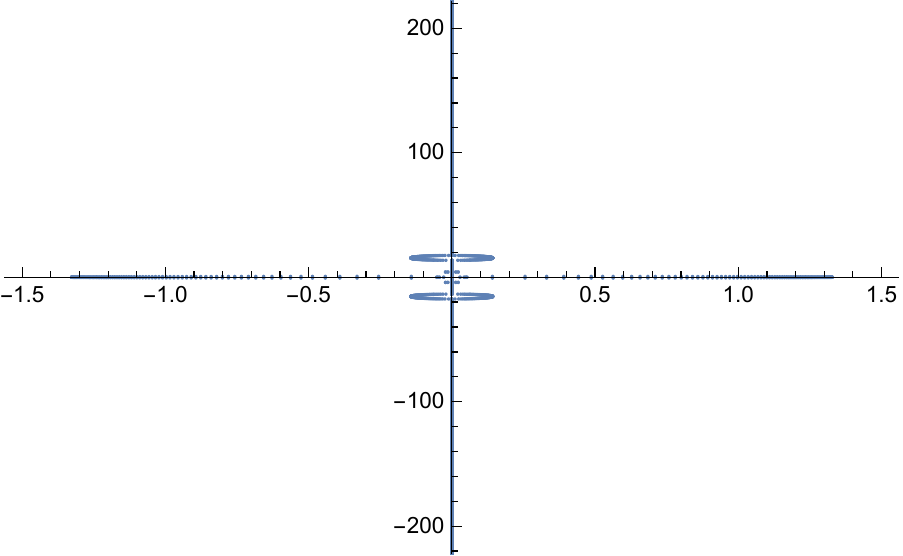}
\includegraphics[width=0.495\textwidth]{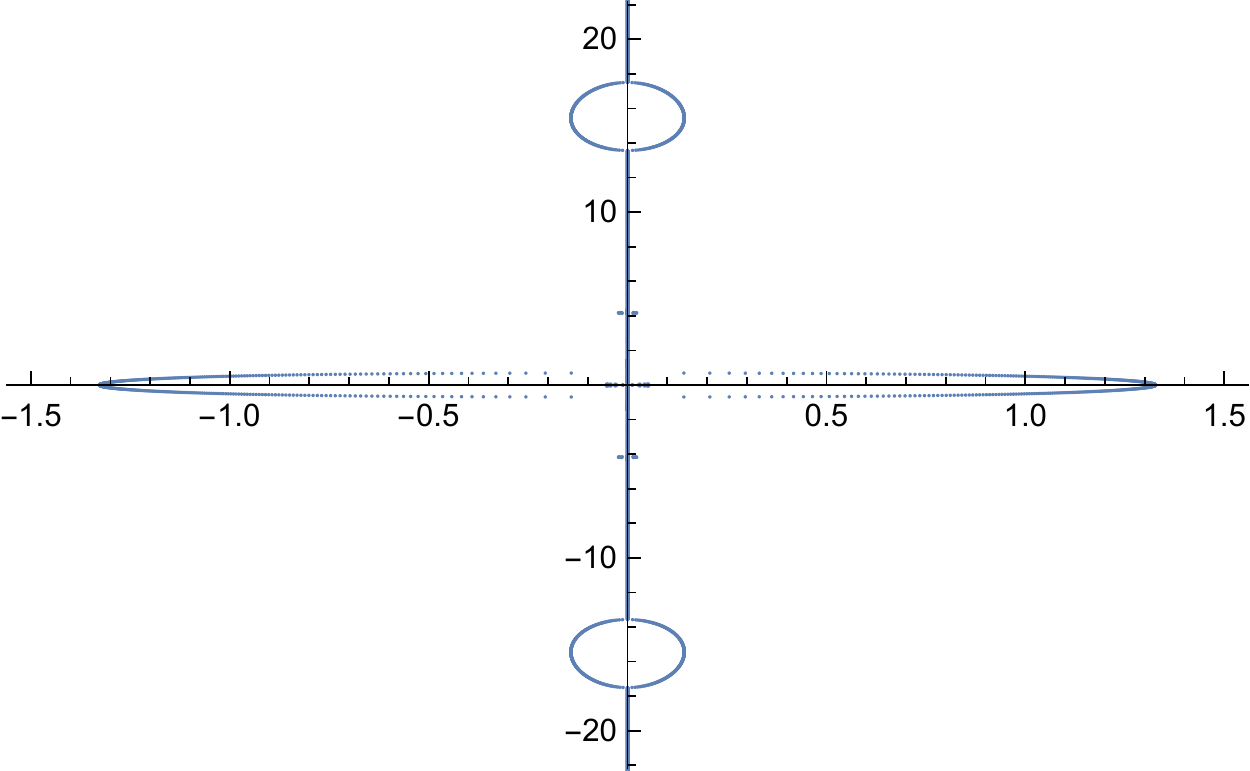}
\caption{The spectrum for \eqref{eqn4:spec} for $c=3$, $b_1=15.97$, $b_2=6.895$, which do not satisfy \eqref{eqn4:cb2}. The spectrum lies along the imaginary axis towards infinity. On the right is a close up for modulational instability near $0\in\mathbb{C}$ and, additionally, finite wavelength instability near $\pm 4.1i$ and $\pm 15.5i$, although the former is hard to distinguish.}
\label{fig4:R3}
\end{figure}

\begin{figure}[htbp]
\centering
\includegraphics[width=0.495\textwidth]{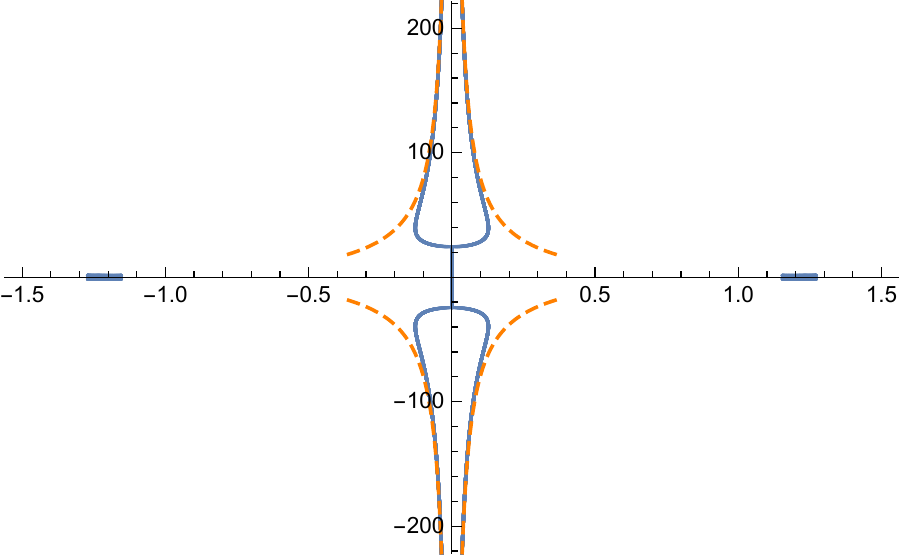}
\includegraphics[width=0.495\textwidth]{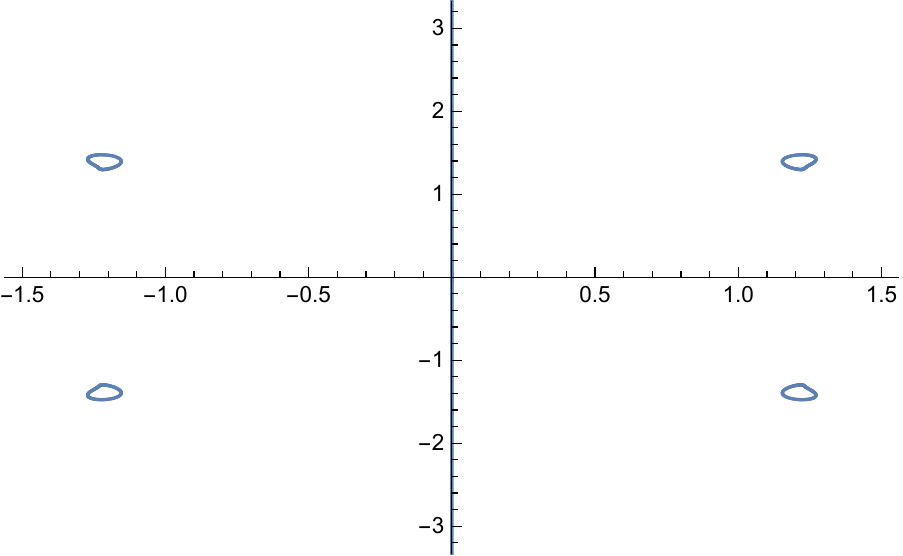}
\caption{The spectrum of \eqref{eqn4:spec} for $c=3$, $b_1=15.97$, $b_2=-1.5$, which satisfy \eqref{eqn4:cb2}. The spectrum leaves the imaginary axis at $\approx\pm24.5i$ and tends towards infinity along the dashed curves $\pm 1.80099 k^{-1}\pm(4.44989k - 18.4791 k^{-1})i$, $k\in\mathbb{R}$. On the right is a close up for {\em no} modulational instability near $0\in\mathbb{C}$, but rather four closed loops of the spectrum near $\pm 1.2 \pm 1.4 i$.}
\label{fig4:R3spine}
\end{figure}

Figures~\ref{fig4:R3} and \ref{fig4:R3spine} show two examples of the spectrum, for which $b_1$ and $b_2$ are in the region~$3$ of Figure~\ref{fig:PhaseDiagram}. In Figure~\ref{fig4:R3}, $b=15.97$ and $b_2=6.895$, for which the three intervals of admissible wave speeds are $(-\infty,-2.90)$, $(0.217,0.802)$, $(2.23,\infty)$, and we take $c=3$. Since $b_2<0$, \eqref{eqn4:cb2} does not hold true. The numerically computed spectrum lies along the imaginary axis toward infinity. The right panel shows modulational and finite wavelength instability.

In Figure~\ref{fig4:R3spine}, $b_1=15.97$ and $b_2=-1.5$, for which the three intervals of admissible wave speeds are $(-\infty,-2.77)$, $(0.123,1.03)$, $(1.95,\infty)$, and \eqref{eqn4:cb2} holds true for $|c|\lesssim 3.26$. We take $c=3$. The spectrum tends towards infinity along the dashed curve, for which we evaluate \eqref{eqn4:lambda} numerically up to the order of $k^{-1}$. There is no modulational instability near $0\in\mathbb{C}$ but four isolated and closed loops of the spectrum off the imagniary axis. 

\begin{figure}[htbp]
\centering
\includegraphics[width=0.495\textwidth]{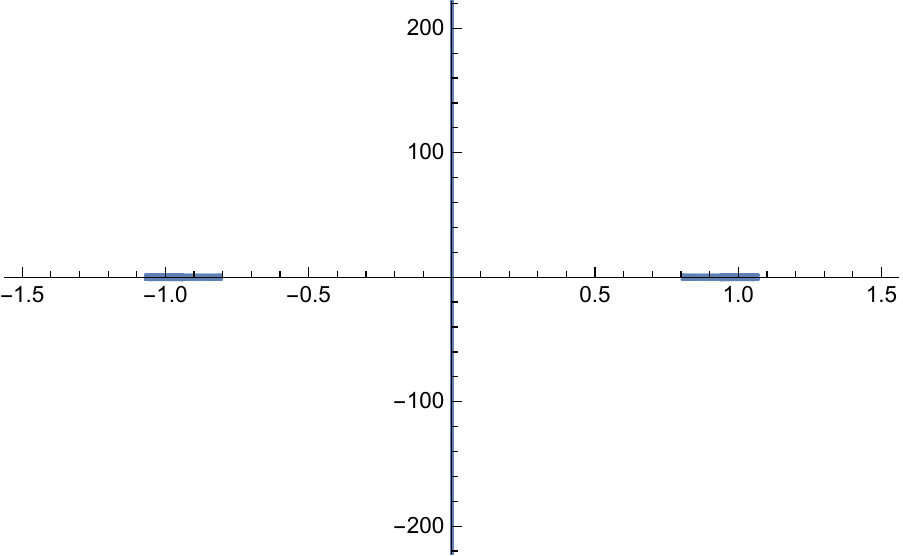}
\includegraphics[width=0.495\textwidth]{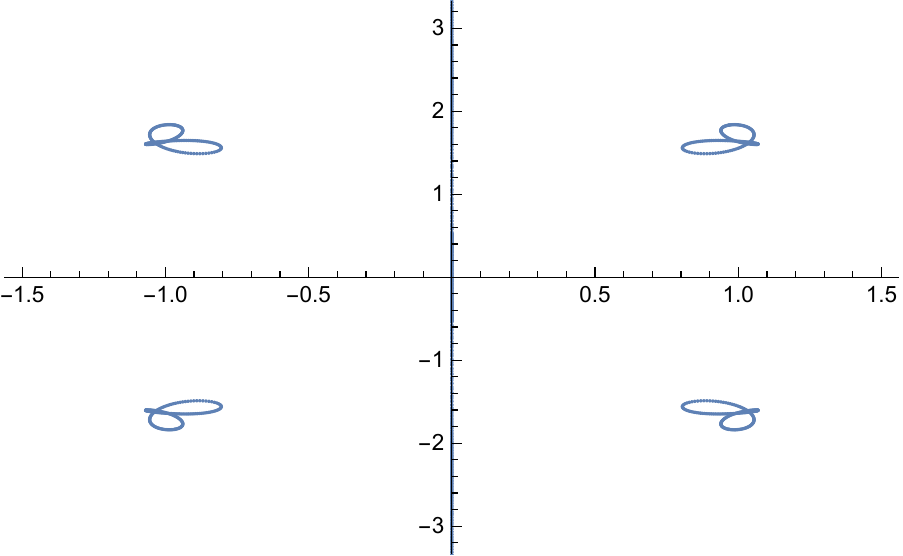}
\caption{The spectrum of \eqref{eqn4:spec} for $b_1=15.97$, $b_2=-1.5$, $c=4$, which do not satisfy \eqref{eqn4:cb2}. The spectrum lies along the imaginary axis towards infinity. On the right is a close-up for {\em no} modulational instability near $0\in\mathbb{C}$, but four isolated and closed loop of the spectrum off the imaginary axis.}
\label{fig4:R3nospine}
\end{figure}
 
Figure~\ref{fig4:R3nospine} shows the spectrum for $b_1=15.97$ and $b_2=-1.5$, the same as in Figure~\ref{fig4:R3spine}, but $c=4$, greater than Figure~\ref{fig4:R3spine}, for which \eqref{eqn4:cb2} no longer holds true. The spectrum lies along the imaginary axis far away from $0\in\mathbb{C}$. On the right is a close up for isolated and closed loops of the spectrum off the imaginary axis. There is no modulational instability.

\begin{figure}[htbp]
\begin{center}
\includegraphics[width=0.495\textwidth]{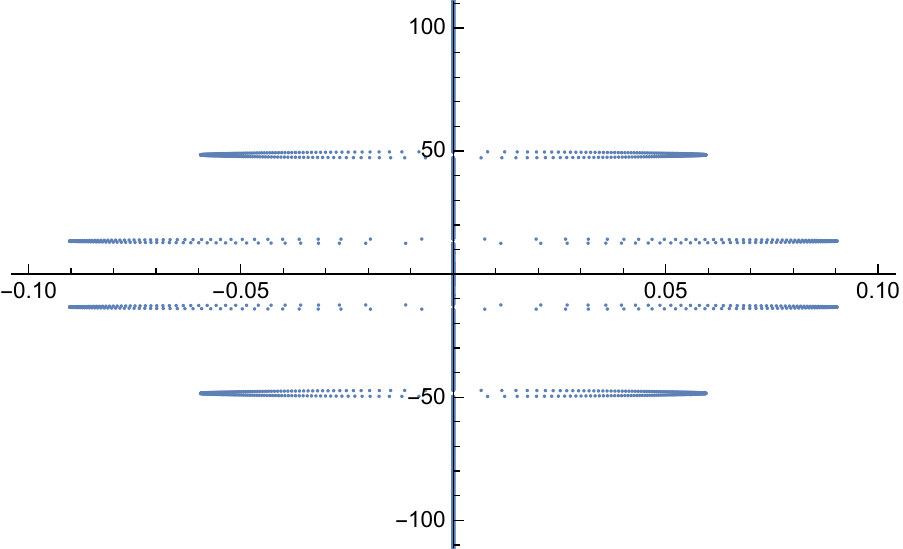}
\caption{The spectrum of \eqref{eqn4:spec} for $c=4$, $b_1=0$, $b_2=40$, which do not satisfy \eqref{eqn4:cb2}. The spectrum lies along the imaginary axis, other than finite wavelength instability near $\pm 12i$ and $\pm 50i$. Notice no modulational instability.}
\label{fig4:R2}
\end{center}
\end{figure}

Last but not least, Figure~\ref{fig4:R2} shows an example of the spectrum, for which the values of $b_1$ and $b_2$ correspond to the bullet point in the region~$2$ of Figure~\ref{fig:PhaseDiagram}, and $c$ and $b_2$ do not satisfy \eqref{eqn4:cb2}. The spectrum lies along the imaginary axis, other than finite wavelength instability. There is no modulational instability.

\bibliographystyle{amsplain}
\bibliography{reference}

\begin{appendix}

\section{The spectrum for the generalized KdV equation}\label{appn}

Let $\theta\in[0,1]$ and we consider the spectral problem
\begin{equation}\label{eqnA:spec}
\lambda\phi=\phi_{xxx}-c\phi_x+\theta(u(x)\phi)_x, \quad 
\phi\in L^2(\mathbb{R}/2\pi\mathbb{Z}),
\end{equation} 
for some smooth and $2\pi$ periodic function $u$ for some $c\neq-1,\in\mathbb{R}$. We wish to show that the spectrum of \eqref{eqnA:spec} lies in the imaginary axis outside some bounded set. Let 
\[
\phi(x)=\sum_{k\in\mathbb{Z}} \widehat{\phi}_k e^{ikx}\quad\text{and}\quad
u(x)=\sum_{k\in\mathbb{Z}} \widehat{u}_k e^{ikx},
\]
and we can reformulate \eqref{eqnA:spec} equivalently as
\begin{equation}\label{eqnA:spec(k)}
\widehat{\phi}_k+\frac{i\theta k}{\lambda+ik^3+ick}\sum_{n\in\mathbb{Z}}\widehat{u}_n\widehat{\phi}_{n-k}=0,
\quad k\in\mathbb{Z}.
\end{equation}
Let
\[
S_k=\Big\{\lambda:|\lambda+ik^3+ick|\leq k\sum_{n\in\mathbb{Z}}|\widehat{u}_n|\Big\}
\quad\text{and}\quad S=\bigcup_{k\in\mathbb{Z}}S_k. 
\]
We pause to remark that $\sum_{n\in\mathbb{Z}}|\widehat{u}_n|<+\infty$ because $u$ is smooth, whence $|\widehat{u}_n|\to0$ as $|n|\to\infty$ more rapidly than polynomially. Observe that:
\begin{itemize}
\item The spectrum of \eqref{eqnA:spec} is contained in $S$ because if $\lambda\notin S$ then \eqref{eqnA:spec(k)} is invertible of $\widehat{\phi}_k$ for any $k\in\mathbb{Z}$ for any $\theta\in[0,1]$; 
\item $S_k$ are disjoint when $|k|\gg1$ because $|\omega(k+1)-\omega(k)|\to\infty$ as $|k|\to\infty$, where $\omega(k)=-k^3-ck$ is the dispersion relation of \eqref{eqnA:spec}; For a regularized long-wave model, such as \eqref{eqn1:rBou}, on the other hand, $|\omega(k+1)-\omega(k)|$ remains bounded for all $k\in\mathbb{R}$;
\item $\lambda$ is continuous in $\theta\in[0,1]$ and there is one in each $S_k$ when $\theta=0$, whereby there is exactly one in each $S_k$ when $\theta=1$.
\end{itemize}

To recapitulate, there are countably many $S_k$ outside of a bounded set, each of which contains one eigenvalue of \eqref{eqnA:spec}. Since \eqref{eqnA:spec} remains invariant under 
\[
\lambda\mapsto \lambda^*\quad\text{and}\quad \phi\mapsto \phi^*
\]
where the asterisk denotes complex conjugation, and under
\[
\lambda\mapsto -\lambda\quad\text{and}\quad x\mapsto -x,
\]
if there is one eigenvalue in each $S_k$ then such an eigenvalue must lie in the imaginary axis. Therefore only finitely many eigenvalues can fall off the imaginary axis. The same will hold true for the quasi-periodic boundary condition.   

\end{appendix}

\end{document}